\documentclass[a4paper,12pt]{article}

\usepackage{amsmath,amsthm,amsfonts,mathtools,amsthm,hyperref,dsfont, enumerate, verbatim, graphicx, epstopdf, enumitem, cancel, subcaption, amssymb}
\usepackage{graphicx}
\usepackage[dvipsnames]{xcolor}

\newtheorem{thm}{Theorem}
\newtheorem{lemma}[thm]{Lemma}
\newtheorem{prop}[thm]{Proposition}
\newtheorem{cor}[thm]{Corollary}
\newtheorem{rem}{Remark}
\newtheorem{defn}[thm]{Definition}
\newtheorem{ass}{Assumption}

\newcommand{\sT}{\mathtt{T}}
\newcommand{\sP}{\mathtt{P}}
\newcommand{\sG}{\mathtt{G}}

\newcommand{\sm}{\mathtt{m}}

\newcommand{\bP}{\mathtt{P}}

\renewcommand{\d}{{\rm d}}
\newcommand{\e}{{\rm e}}

\newcommand*{\norm}[1]{\lVert #1 \rVert}

\newcommand{\cF}{\mathcal{F}}

\newcommand{\cN}{\mathcal{N}}
\newcommand{\cZ}{\mathcal{Z}}
\newcommand{\cE}{\mathcal{E}}
\newcommand{\cP}{\mathcal{P}}

\newcommand{\us}{\mathbf{s}}

\newcommand{\uv}{\mathbf{v}}

\everymath{\displaystyle}

\makeatletter

\title{Many-to-few for non-local branching Markov process}

\author{Simon C. Harris\thanks{University of Auckland. E-mail: \texttt{simon.harris@auckland.ac.nz}}, \ Emma Horton\thanks{Inria Research Centre Bordeaux. E-mail: \texttt{emma.horton@inria.fr}},   
\ Andreas E. Kyprianou\thanks{University of Bath. E-mail: \texttt{a.kyprianou@bath.ac.uk}} 
\ and
Ellen Powell\thanks{University of Durham. E-mail: \texttt{ellen.g.powell@durham.ac.uk}}
 }

\begin{document}
\maketitle
\begin{abstract}
We provide a many-to-few formula in the general setting of non-local branching Markov processes. This formula allows one to compute expectations of $k$-fold sums over functions of the population at $k$ different times. The result generalises \cite{Many2few} to the non-local setting, as introduced in \cite{YaglomNTE} and \cite{moments}. 
As an application, we consider the case when the branching process is critical, and conditioned to survive for a large time. In this setting, we prove a general formula for the limiting law of the death  time of the most recent common ancestor of two particles selected uniformly from the population at two different times, as $t\to\infty$. Moreover, we describe the limiting law of the population sizes at two different times, in the same asymptotic regime.
%we consider choosing With this in hand, we provide some initial results concerning the spatial behaviour of 2-spine asymptotics, in which the dependency on space becomes largely irrelevant. 
\medskip


\medskip

\end{abstract}

\section{Introduction}

\subsection{Main results}

Our main result, the so called {\it many-to-few formula}, is a way to rewrite the expectation of a general $k$-fold sum, depending on the entire configuration of a   branching Markov process  at $k$ different times, as an expectation with respect to the behaviour of $k$ distinguished lines of descent under a tilted measure. %Similar tools have been employed to great effect in, for example, \cite{AH, AHbook, YaglomNTE, moments}. 
We generalise the original and well cited main result of  \cite{Many2few}, by allowing for non-local branching, and not requiring the  $k$  individuals to be sampled at the same time.

\medskip

The many-to-few formula generalises the role of the classical spine decomposition for spatial branching processes, which converts expectation identities for additive functionals of spatial branching processes to Feynman-Kac formulae for a single Markov particle trajectory.  The latter has proved to be an important tool in analysing the growth and spread of a rich variety of  branching Markov processes and related models; see for example the monographs \cite{AHbook, Janos, ZS, HK, Bertoinbook} among a wide base of research literature that is to too extensive to exhaustively list here. 
The many-to-few formula has already played an important and similar role to the classical spine decomposition  as a tool to interrogate various questions pertaining to particle correlation that arise in e.g. genealogical coalescent structure, \cite{HJR}, martingale convergence, \cite{SNTE-II}, maximal displacement of extreme particles, \cite{LTZ}, the structure of level sets  for branching Brownian motion, \cite{CHO} and the analysis of certain models from the theory of stochastic genetics, \cite{FRS}. We further remark that the many-to-few formula is  related to recent work pertaining to asymptotic moment convergence in \cite{moments}. 

 % The latter of these two generalisations makes the many-to-few formula better suited to 

\medskip

 %{\color{blue} **history/how this has been useful before: AEK k-moments, SImon's paper, citations of Simon's paper and classic spines**}. 
 
We refrain from attempting to give a precise statement of the many-to-few formula here, deferring instead to Lemma \ref{m2k-ktimes} below,  as we will need to introduce several objects in order for the formula to be understood in a meaningful way. %However, let us point out already that the result holds for the entire class of non-local branching processes introduced in \cite{YaglomNTE} and \cite{moments}. {\color{blue} **explain how it extends previous results - relating back to history**}
We note that, simultaneously to the results we present here a general  branching Markov process setting,  similar ideas have been developed in \cite{FRS}.

\medskip 

Our main motivating application %for the many-to-few formula, Lemma \ref{m2k-ktimes},
 is to understand the limiting genealogy of a so-called \emph{critical} branching Markov process, when  conditioned to survive for an arbitrarily  long time. Our second main result, Proposition \ref{P:split2}, is a general statement in this direction. More precisely, for a critical non-local branching Markov process conditioned to survive until a large time $t$, we provide a precise asymptotic for the death time of the most recent common ancestor of two individuals sampled uniformly from the population at two different times.  In Proposition \ref{prop:joint-conv}, we also describe the limiting law of the population sizes at two different times, in the same asymptotic regime.

\subsection{Set-up and assumptions}\label{SS:set-up}

Let $E$ be a Lusin space. Throughout, will write $B(E)$ for the Banach space of bounded measurable functions on $E$ with norm $\norm{\cdot}$, $B^{+}(E)$ for non-negative bounded measurable functions on $E$ and $B^{+}_1(E)$ for the subset of functions in $B^{+}(E)$ which are uniformly bounded by unity.

\medskip

We consider a spatial branching process in which, given their point of creation, particles evolve independently according to a Markov process, $(\xi, \mathbf{P})$, which can be characterised via the semigroup $\sP_t[f](x) = \mathbf{E}_x[f(\xi_t)]$, for $x \in E$, $t \ge 0$ and $f \in B^{+}_1(E)$. In an event which we refer to as `branching', particles positioned at $x$ die at rate $\beta(x)$ where $\beta\in B^+(E)$ and instantaneously, new particles are created in $E$ according to a point process. The configurations of these offspring are described by the random counting measure
\begin{equation}
\mathcal{Z}(A) = \sum_{i = 1}^N \delta_{x_i}( A), 
\label{Z}
\end{equation}
for Borel $A$ in $E$. The law of the aforementioned point process may depend on $x$, the point of death of the parent, and we denote it by $\mathcal{P}_x$, $x\in E$, with associated expectation operator given by $\mathcal{E}_x$, $x\in E$.  This information is captured in the so-called branching mechanism
\begin{equation}
  \sG[f](x) :=  \beta(x)\mathcal{E}_x\left[\prod_{i = 1}^N f(x_i) - f(x)\right], \qquad x\in E, \, f\in B^+_1(E).
  \label{linearG}
\end{equation}
Without loss of generality we can assume that $\mathcal{P}_x(N =1) = 0$ for all $x\in E$ by viewing a branching event with one offspring as an extra jump in the motion. On the other hand, we do allow for the possibility that $\mathcal{P}_x(N =0)>0$ for some or all $x\in E$.

\medskip

Moreover, we do not need $\bP$ to have the Feller property, and it is not necessary that $\bP$ is conservative. That said, if so desired, we can append  a cemetery state $\{\dagger\}$ to $ E$, which is to be treated as an absorbing state, and regard $\bP$ as conservative on the extended space $E\cup\{\dagger\}$, which can also be treated as a Lusin space. Equally, we can extend $\sG$ to $E\cup\{\dagger\}$ by defining it to be zero on $\{\dagger\}$, that is no branching activity on the cemetery state. 

\medskip

Henceforth we refer to this spatial branching process as a $(\sP, \sG)$-branching Markov process. It is well known that if we arbitrarily choose an order for the particles at each branching event then we may denote the configuration of particles at time $t$ by $\{x_1(t), \ldots, x_{N_t}(t)\}$ (where $N_t$ denotes the number of particles alive at time $t$) and we can also associate an Ulam-Harris label \[v_i(t)\in \Omega:=\{\emptyset\}\cup \bigcup_{n\ge 1} \mathbb{N}^n\] to the $i$-th particle at time $t$. For $v, w \in \Omega$, we write $v \preceq w$ to mean that $v$ is an ancestor of $w$, which means there exists $u \in \Omega$ such that $vu = w$. {Moreover, we write $v \prec w$ to mean that $v \preceq w$ in the strict sense, that is, the possibility that $v=w$ is excluded.} We say that $v, w \in \Omega$ are siblings if there exists $u \in \Omega$ and $i \neq j$ such that $v = ui$ and $w = uj$.

\medskip

The branching Markov process can be described via the co-ordinate process $X= (X_t, t\geq0)$ in the space of {counting} measures on $E\times \Omega$ with non-negative integer total mass, denoted by $M(E\times \Omega)$, where
\[
X_t (\cdot) = \sum_{i =1}^{N_t}\delta_{(x_i(t),v_i(t))}(\cdot), \qquad t\geq0.
\]
In particular, $X$
is Markovian in  $M(E\times \Omega)$. Its probabilities will be denoted $\mathbb{P}: = (\mathbb{P}_{\delta_x}, x\in E)$ where for $x\in E$, $\mathbb{P}_{\delta_x}$ denotes the law of the process starting from $\delta_{(x,\emptyset)}\in M(E\times \Omega)$. 
%With this notation in hand, %it is worth noting that the independence that is manifest in the definition of branching events and movement implies that 
%if 
%We define, 
%\begin{equation}
%\sv_t[f](x) = \mathbb{E}_{\delta_x}\left[\prod_{i = 1}^{N_t} f(x_i(t))\right], \qquad f\in B^+_1(E), \, t\geq 0, \, x \in E,
%\label{nonlin}
%\end{equation}
%then for $\mu\in M(E)$ given by $\mu = \sum_{i =1}^n\delta_{y_i}$, we have
%\begin{equation}
%\label{MBP}
%\mathbb{E}_{\mu}\left[\prod_{i = 1}^{N_{t}} f(x_i(t))\right] = \prod_{i = 1}^{n}\sv_t[f](y_i), \qquad t\geq 0.
%\end{equation}
%then for $f\in B_1^+(E)$ and  $x\in E$, 
%\begin{equation}
%\sv_t[f](x) = \hat\sP_t[f](x) + \int_0^t \sP_s\left[ \sG[\sv_{t-s}[f]]\right](x)\d s, \qquad t\geq0, 
%\label{nonlinv}
%\end{equation}
%where $\hat\sP_t$ is a slight adjustment of $\sP$, which returns the value $1$ on the event that $\xi$ is killed. 

%\smallskip

%The proof is classical  and follows standard reasoning for semigroup integral equations e.g. as  in  \cite{SNTE-I, SNTE-II}: First  conditioning $\sv$ on the time of the first branching event, then using the principle of transferring between multiplicative and additive potentials  in the resulting integral equation (cf. Lemma 1.2, Chapter 4 in \cite{Dynkin2}) shows that \eqref{nonlinv} holds. Gr\"onwall's Lemma and the fact that $\beta\in B^+(E)$ ensure that the relevant integral equations have unique solutions. 

\medskip

%Recalling that $\beta \in B^+(E)$, 
Under the additional assumption that $\mathrm{sup}_{x \in E} \,\mathcal{E}_x(N) < \infty$, where we recall that $N$ is the (random) number of offspring produced at a branching event, we define the linear semigroup
\[
\sT_t[f](x) := \mathbb{E}_{\delta_x}[X_t[f]] := \mathbb{E}_{\delta_x}\left[\sum_{i = 1}^{N_t} f(x_i(t))\right], \quad f\in B^+(E).
\]
%can be written in an alternative form via a many-to-one formula. Indeed, suppose that we define 
%\[
%\sm[f](x) := \mathcal{E}_x\left[\sum_{i = 1}^N f(x_i)\right], \quad B(x) := \beta(x)(\sm[1](x) - 1)
%\]
%and introduce a new Markov process $\hat{\xi} = (\hat{\xi}_t, t\geq0)$ which evolves as the process $\xi$ but at rate $\beta(x)\sm[1](x)$ the process is sent to a new position in $E$, such that for all Borel $A\subset E$, the new position is in $A$ with probability  $\sm[\mathbf{1}_A](x)/\sm[1](x)$.  We will refer to the latter as {\it extra jumps}. Note the law of the extra jumps  is well defined thanks to the assumption that $\sup_{x \in E}\sm[1](x) < \infty$. Accordingly we denote the probabilities of $\hat\xi$ by $(\hat{\mathbf{P}}_x, x\in E)$. Then the many-to-one formula is as follows.

%\begin{lemma}\label{M21} 
%For $f\in B^+(E)$, $t\geq 0$ and $x \in E$, we have 
%\begin{equation}
%\sT_t [f](x) = \hat{\mathbf{E}}_x\left[\exp\left(\int_0^t B(\hat\xi_s)\d s\right) f(\hat\xi_t)\right].
%\label{firstextrajump}
%\end{equation}
%\end{lemma}

%Again, the proof is fairly standard and we refer the reader to \cite{SNTE-I, SNTE-II} for how to prove the above result. 

%\medskip

%There is also another many-to-one formula for $\sT_t$ given via the so-called spine decomposition. 

Now let us introduce an assumption, that we will use throughout the article unless stated otherwise. We will use the notation 
\[
\langle f, \mu \rangle := \int_E f(x)\mu(\d x), \quad f \in B(E), \mu \in M(E),
\]
where $M(E)$ is the set of finite measures on $E$.

\begin{ass}\label{ass1}For the Markov process $(\xi,\mathbf{P})$, we assume that 
\begin{enumerate}
\item[(a)] it  admits a c\`adl\`ag modification;
\item[(b)] there exists an eigenvalue $\lambda \in \mathbb R$ and a corresponding right eigenfunction $\varphi \in B^+(E)$ and finite left eigenmeasure $\tilde\varphi$ such that, for $f\in B^+(E)$,
\[
\langle \sT_t[\varphi] , \mu\rangle = {\rm e}^{\lambda t}\langle{\varphi},{\mu}\rangle 
\text{ and } 
\langle{\sT_t[f] },{\tilde\varphi} \rangle= {\rm e}^{\lambda t}\langle f, {\tilde\varphi}\rangle,
\]
 for all $\mu\in M(E)$. 
\end{enumerate}
\end{ass}

The first part of the above assumption is a regularity assumption on the Markov process $(\xi,\mathbf{P})$, which ensures that we can use the theory of martingale changes of measure. The second is a Perron Frobenius assumption that ensures the existence of the leading eigenvalue and corresponding eigenfunctions.

\subsection{Outline}
The many-to-few formula, Lemma \ref{m2k-ktimes}, will ultimately allow us to express  expectations of general $k$-fold sums depending on the entire branching process under $\mathbb{P}$, in terms of an expectation with respect to $k$ so-called \emph{spine particles} under a different measure $\mathbb{Q}^k$. In Section \ref{S:com}, we define this measure $\mathbb{Q}^k$, introduce the notion of spines, and give an explicit expression for the Radon-Nikodym derivative of $\mathbb{Q}^k$, with respect to the original law of our branching process plus  some uniformly chosen marked lines of descent (we call this measure $\mathbb{P}^k$). We then state and prove our main result, the many-to-few formula, in Section \ref{S:m2few}. We also explain some special cases in which the formula simplifies nicely. Section \ref{S:application} is devoted to our main application, which is to derive some two-point asymptotics for the geneologies of critical branching processes, when they are conditioned to survive for a large time $t$. More precisely, we provide an asymptotic (as $t\to \infty$) for the death time of the most recent common ancestor of two particles sampled uniformly from the population at two different times.

\section{Spines, martingales and changes of measure}\label{S:com}

In this section, we introduce two measures under which  our class of  branching Markov process additionally  identifies $k$ distinguished genealogical lines of descent, or {\it spines}. %which may possibly overlap. 
The first  of these two measures  is a simple adaptation of  the original law. The second measure is identified via a change of measure with respect to a certain multiplicative martingale.

% will extend the notion of the many-to-few representation for the $k$-th moment of the branching process, as developed in~\cite{Many2few}, to the class of processes introduced in the previous section. {\color{blue} **I'm not sure the previous sentence makes much sense?**} We will first extend the notion of branching processes to branching processes with marked individuals or {\it spines}. 

\subsection{Definition of the measures $\mathbb{P}^k$ and $\mathbb{Q}^k$}\label{SS:pkqk}

\subsubsection{Definition of $\mathbb{P}^k$}

We first introduce a measure $\mathbb{P}^k$ on the set of processes taking values in $M(E\times \Omega \times \mathcal{P}(\{1, \ldots, k\}))$, the space of counting measures on $E\times \Omega \times \mathcal{P}(\{1, \ldots, k\})$,  where {$\mathcal{P}(\{1, \ldots, k\})$ is the set of subsets of $\{1,\dots, k\}$.} %At time $t \ge 0$, an atom at $(x, v, \{b_1, \ldots, b_j\})$ means there is a particle at position $x \in E$ with label $v \in \Omega$ carrying $j$ marks, which are denoted by $\{b_1, \ldots, b_j\}$.
For convenience, let $\tilde X$ denote the branching process on the space $M(E \times \Omega \times \mathcal{P}(\{1, \ldots, k\}))$, that is 
\[
  \tilde X_t = \sum_{i = 1}^{N_t} \delta_{(x_i(t), v_i(t), \mathbf{b}_i(t))},
\]
where $\mathbf{b}_i(t)\in\mathcal{P}(\{1,\dots, k\})$ denotes the set of marks carried by the $i$-th particle alive at time $t$. Whenever $\mathbf{b}_i(t)\neq\emptyset$, we refer to the individual $i$ as a {\it spine}. In that case, we say that the spine {\it carries $|\mathbf{b}_i(t)|$ marks}.
Given $\tilde X$, define $X$ to be its projection onto $M(E \times \Omega)$.  
With this notation in hand, we let $(\cF_t, t \ge 0)$ denote the natural filtration generated by $X$ and
$(\cF_t^k, t \ge 0)$ denote the natural filtration generated by $\tilde X$.

\medskip

Then, we have the following description of the measure $\mathbb{P}^k$.

\begin{defn} The construction of $\tilde{X}$ under $\mathbb{P}^k$ goes as follows.
\begin{enumerate}
\item We start with a single particle at $x \in E$ which carries $k \ge 1$ marks.
\item All particles move according to the semigroup $\sP$, independently of each other given their birth times and configurations.
\item Let $\xi_t^i$ denote the position of the particle that carries the mark $1 \le i \le k$ at time $t \ge 0$.
\item A particle at $y \in E$ {carrying $j$ marks}  $b_1 < b_2 < \dots < b_j$, dies at rate $\beta(y)$ and simultaneously produces a random number of new particles according to $(\cZ, \mathcal{P}_y)$. The $j$ marks each choose a particle to follow independently and uniformly from the $N = \langle 1, \cZ\rangle$ available particles.
\item In the event that a particle carrying  $j > 0$ marks dies and is replaced by $0$ offspring, it is sent to the cemetery state, along with its marks.
\end{enumerate}
\end{defn}
\noindent Note that the above definition of $\mathbb{P}^k$ is such that $X$ has the same law under both $\mathbb{P}^k$ and $\mathbb{P}$. 

\subsubsection{The measure $\mathbb{Q}^k$}

We will now introduce a second measure, $\mathbb{Q}^k$, under which particles not carrying any marks evolve { in the same manner as particles under $\mathbb{P}$}, while spines (that is, particles carrying marks) evolve differently. {In the next section we will show that $\mathbb{Q}^k$ can be defined via a change of measure of $\mathbb{P}^k$.} Before describing the process, we {need to} introduce some more notation. 

\medskip

Suppose that we are given a functional $(\zeta(\cdot, t), t \ge 0)$ such that $\zeta(\xi, t)$ is a non-negative unit-mean martingale with respect to the natural filtration of the Markov process $(\xi_t, t \ge 0)$ with semigroup $(\bP_t, t \ge 0)$. We will assume that $\zeta$ takes the value $0$ whenever $\xi_t = \dagger$.
Now for $k,n\in \mathbb{N}$, define 
\begin{equation}\label{eq:new_offspring}
\langle \varphi, \cZ \rangle_{k, n} ={\bf 1}_{(n\leq N)} \sum_{[k_1, \dots, k_N]_k^n}{{k}\choose{k_1, \dots, k_N}}\prod_{i : k_i > 0}\varphi(x_i),
\end{equation}
where $(x_i, i = 1, \dots, N)$ are as in \eqref{Z} and $[k_1, \dots, k_N]_k^n$ is the set of non-negative integer $N$-tuples $(k_1, \dots, k_N)$ such that $k_1 + \dots + k_N = k$ and exactly $n$ of the $k_i$s are positive. If $\mathcal{Z}$ corresponds to the offspring of a particle carrying $k$ marks, one can think of a single term in the sum $\langle \varphi, \cZ\rangle_{k, n}$ as a weight associated to the event that the $k$ marks are distributed among the offspring, by giving exactly $k_i$ marks to the $i$th offspring particle, $i=1,\ldots, N$ (see below for a more precise interpretation).
With this notation in hand, now define
\begin{align}
\langle \varphi, \cZ \rangle_k(x) &:= \sum_{1\le n \le k}   \varphi(x)^{-n} \langle \varphi, \cZ \rangle_{k, n}
%&:=\sum_{1\le n \le k} \varphi(x)^{-n}  \sum_{k_1+\dots + k_n=k} \sum_{i_1\ne i_2\ne \dots \ne i_n} \varphi(x_{i_1})\dots \varphi(x_{i_n}) 
\label{multi-offspring}
\end{align}
and set {${\sm}_{k}(x)=\cE_x(\langle \varphi,\cZ\rangle_k(x))$}.
Note that in the case of local branching, {$ \langle \varphi, \cZ \rangle_k \equiv N^k$} for any $x \in E$.

%Let 
%\[
%\sm_j[f](x) := \mathcal{E}_x[\langle f, \mathcal{Z}\rangle^j]
%\] 
%denote the $j$-th moment of the offspring distribution and 
%\[
%\alpha_j(x) := \frac{\beta(x)}{\varphi^j(x)}(\sm_j[\varphi](x) - \varphi^j(x)).
%\]

%We would now like to extend the notion of the spine introduced in the previous section by defining a new measure under which the spines $\psi^1, \dots, \psi^k$ now have a biased motion, altered branching rate and biased offspring distribution. To this end, let us introduce the measure $\mathbb{Q}_{\delta_x}$ via the following construction. 

\begin{defn} The construction of $\tilde{X}$ under $\mathbb{Q}^k$ goes as follows.

\begin{enumerate}
\item Again, we begin with one particle at $x \in E$ carrying all the marks $\{1, \dots, k\}$. In what follows, particles carrying marks are referred to as spines.
%\item We let $\xi_t^i$ denote the position of the particle with mark $i$ at time $t$, $i = 1, \dots, k$. 
\item Any spine (that is, any particle carrying any number of marks) moves according to the semigroup
\begin{equation}
 \sP_t[{g}](x) :=\frac{1}{\zeta(\xi, 0)}\mathbf{E}_{x}\left[\zeta(\xi, t)g(\xi_t) \right], \qquad x \in E,
 \label{k-motion-bias}
\end{equation}
where $(\xi_s,s\geq0)$ denotes the motion.
\item Suppose a spine carries marks $\mathbf{b} = (b_1,\dots,b_j)$. Then for each $1 \le n \le j$, an independent exponential clock rings at rate $\beta(x) {\sm}_{j,n}(x),$ %into $n$ distinct spines at rate
\[
{\sm}_{j,n}(x): = \varphi(x)^{-n} \mathcal{E}_x(\langle \varphi, \cZ\rangle_{j, n}), \qquad x\in E.\]
When the first of these clocks rings, a branching event occurs, and if it is the $n$th clock, the $j$ marks carried by the parent will be given to exactly $n$ distinct offspring particles. 
\medskip 

More precisely, if the first clock to ring is the $n$th one, the positions of the offspring are described by $\cZ$ with law $\cP_x^{(j, n)}$ defined by
\begin{equation}
 \frac{\d \cP_x^{(j, n)}}{\d \cP_x}:=\frac{\langle \varphi, \cZ\rangle_{j, n}}{\mathcal{E}_x(\langle \varphi, \cZ\rangle_{j, n})}. 
\label{bias-offspring}
\end{equation}
Then given $\mathcal{Z}$, for each $(k_1, \dots, k_N) \in [k_1, \dots, k_N]^n_{j}$, the probability that the $i$th offspring particle receives exactly $k_i$ marks for each $1\le i \le N$, is given by 
$$\frac{ {{k}\choose{k_1, \dots, k_N}} \prod_{i : k_i > 0}\varphi(x_i)}{\langle \varphi, \cZ\rangle_{j, n}}.$$
On this event, the way that the marks $b_1, \dots, b_j$ are distributed among the offspring is that such that any valid configuration (that is, satisfying the constraint that exactly $k_i$ marks are given to offspring particle $i$ for each $1\le i \le N$) has the same probability: $$\frac{1}{{{k}\choose{k_1, \dots, k_N}}}. $$
%{\color{blue}AEK: This is really clunky -  somehow needs to  re-write the birthing logic in 3 mentioning that, given a partition of spinal the spinal mass $j$ according to $(k_1, \dots, k_N) \in [k_1, \dots, k_N]^n_{\color{red}j}$, if the mark of the parent is given by $\mathbf{b} = (b_1,\dots,b_j)$, then they are distributed among the $n\leq j$ sub-blocks of positive spinal mass in an exchangeable (uniform) manner. The combinatoric of the latter is somehow mixed up with \eqref{eq:new_offspring}?
%}

\item Particles that do not carry marks issue independent copies of $(X, \mathbb{P})$. Marked particles then continue from Step 2.  

\end{enumerate}
\end{defn}

\begin{rem}\label{rem:alt-description}\rm
An alternative description of the third step above, is in terms of the \emph{total} branching rate. Namely, suppose a spine carries marks $\mathbf{b} = (b_1,\dots,b_j)$. Then it branches at rate $\beta(x) \sm_j(x)$, and on such a branching event, the offspring  positions are described by $\mathcal{Z}$ whose law is weighted (with respect to $\mathcal{P}_x$) by $\langle \varphi, \mathcal{Z}\rangle_j / \sm_j(x)$. Moreover, given $\mathcal{Z}$ (and the position $x$ of the spine before branching), any particular allocation of the marks $b_1,\ldots, b_j$ among the $N$ offspring (that is, a partition $S_1, \dots, S_N$ of $\{b_1, \dots, b_j\}$) has probability equal to $\langle \varphi, \mathcal{Z} \rangle_j^{-1} \textstyle \prod_{i: |S_i|>0} (\varphi(x_i)/\varphi(x)) $.
However, we prefer to highlight the decomposition according to $n$ (as given in the description of the measure above), since this will be the one we use in practice.
\end{rem}

\begin{rem}
	\rm
There are many different variants of the measure $\mathbb{Q}^k$ that we could have described,  that would also be related to  $\mathbb{P}^k$ by a martingale change of measure, and would also lead to a many-to-few type formula. Our specific choice of $\mathbb{Q}^k$ is motivated by our main application: describing the genealogical structure of the branching process when it is conditioned on survival. In particular, when we use our many-to-few formula for this purpose, we get an extremely simple structure - see \eqref{eq:natural_martingale}. Note that in the case of local branching, this measure $\mathbb{Q}^k$ is identical to that which appears in \cite{Many2few}. When $k=1$, and we make a particular choice for $\zeta$, this also agrees with the ``spine decomposition'' given in \cite{YaglomNTE}.
\end{rem}

\subsection{The martingale change of measure}\label{SS:mgale}
We will now explain how the measures $\mathbb{P}^k$ and $\mathbb{Q}^k$ are connected via a change of measure. 
Let us first introduce some further notation.
\medskip 

Given $v \in \Omega$, note that $X_t(E \times \{v\}) = 0$ except on some unique (possibly empty) interval $[\sigma_v, \tau_v)$, on which $X_t(E \times \{v\}) = 1$.  We will often use the notation $\tau_v^-$ for the left limit of $\tau_v$. If $t \in [\sigma_v, \tau_v)$, there exists a unique $X_v(t) \in E$ such that $X_t(X_v(t) \times \{v\}) = 1$ and a unique $\mathbf{b}_v \in \mathcal{P}(\{1, \dots, k\})$ such that $\tilde{X}_t(E \times \{v\} \times \mathbf{b}_v) = 1$ for all $t \in [\sigma_v, \tau_v)$. Heuristically, $\sigma_v$ and $\tau_v$ are the birth and death times of particle $v$, respectively, $X_v(t)$ represents its position at time $t$ during its lifetime, and ${\bf b}_v$ represents the set of marks it carries. We further set {$D_v = |\mathbf{b}_v|$} to be the number of marks carried by the particle with label $v$. For each $v \in \Omega$, let $N_v$ denote the number of offspring produced by $X_v(\tau_v)$.

\medskip

Set $\mathcal{N}_t := \{v \in \Omega : t \in [\sigma_v, \tau_v)\}$ so that $N_t = |\mathcal{N}_t|$. For each $t \ge 0$ and $j = 1, \dots, k$, let $\psi_t^j$ and $\xi_t^j$ denote the unique elements of $\Omega$ and $E$, respectively, such that there exists $\mathbf{b} \in \mathcal{P}(\{1, \dots, k\})$ {with $j \in \mathbf{b}$} and $\tilde{X}_t(\xi_t^j \times \{\psi_t^j\} \times \mathbf{b}) = 1$. Finally, let us define the skeleton at time $t$ to be $S_k(t):= \{\psi_s^1, \ldots, \psi_s^k : s \le t\}$, so that $S_k(t)$ is a subset of labels in the tree. For $v\in S_k(t)$, let $M_v$ denote the number of distinct offspring of $v$ that are given a mark (that is, the number of distinct spine offspring).

\begin{defn}
Define the ${\mathcal{F}}_t^k$-adapted process $({W}_t^k, t \ge 0)$ by
\begin{align}
{W}_t^k 
:=  & %\mathbf{1}_{\{\zeta(X_v,t)>0\, \forall v\in \cN_t\}} 
\prod_{v \in S_k(t)}\frac{\zeta(X_v, \tau_v^- \wedge t)}{\zeta(X_v, \sigma_v)}  
\prod_{v \in S_k(t)}{\rm e}^{-\int_{\sigma_v}^{\tau_v \wedge t}\beta(X_v(s))({\sm}_{D_v}(X_v(s)) - 1)\d s} 
%\prod_{v \in S_k(t) \cap \mathcal{N}_t}{\rm e}^{-\int_{\sigma_v(t)}^{t}\beta(X_v(s))({\sm}_{D_v}(X_v(s)) - 1)\d s} 
\notag\\
		&\times \prod_{v \in S_k(t) \backslash \{\emptyset\}} {\varphi(X_v(\sigma_v))} \prod_{v \in S_k(t) \backslash \mathcal{N}_t}\frac{N_v^{D_v}}{\varphi(X_v(\tau_v^-))^{M_v}}.
		\label{def-Wk}
	\end{align}
	
%Further, define the $\cF_t$-adapted process $Z_t^k$ by \ellen{(here is $D_v=D_{\mathbf{u}}(v)$? Is $\mathcal{Z}_v$ the set of positions of offspring of $v$?)}
%\begin{align}
%  Z_t^k &:= \sum_{u_1, \dots, u_k \in \cN_t}\mathbf{1}_{\{\zeta(\xi^i, t) > 0 \forall i = 1, \dots, k\}}\prod_{v \in S_k(t)}
%\frac{\zeta(X_v, \tau_v^-(t))}{\zeta(X_v, \sigma_v(t))} {\rm e}^{-\int_{\sigma_v(t)}^{\tau_v(t)}\alpha_{D_v}(X_v(s))\d s}\notag \\
%&\qquad \qquad \times \prod_{v \in S_k(t) \backslash  \mathcal{N}_t}{\varphi(X_v(\tau_v^-))^{-D_v}}\prod_{v \in S_k(t) \backslash \{\emptyset\}}\varphi(X_v(\sigma_v))^{D_v}
%\end{align}
\end{defn}

\begin{rem}\rm
	Henceforth, for $s\le t$ and $x, y \in E$, we set $\zeta(x,t)/\zeta(y,s)=1$ whenever $\zeta(x,t)=\zeta(y,s)=0$.
\end{rem}

\begin{rem}\rm
We emphasise that then when a branching event occurs and the spines all choose the same particle to follow, this is included as an element of the skeleton, $S_k(t)$. 
\end{rem}

\begin{rem}\label{rem:alt-mg}\rm
Note that we may equivalently write
\begin{align}
W_t^k
% = &\prod_{v \in S_k(t)}\frac{\zeta(X_v, \tau_v^- \wedge t)}{\zeta(X_v, \sigma_v)}  \notag \\
%& \times \prod_{v \in S_k(t) \backslash \cN_t}{\sm_{D_v}(X_v(\tau_v^-))} {\rm e}^{-\int_{\sigma_v}^{\tau_v}\beta(X_v(s))(\sm_{D_v}(X_v(s)) - 1)\d s}
%%{\varphi(X_v(\tau_v^-))^{D(v)}}
%\prod_{v \in \cN_t}{\rm e}^{-\int_{\sigma_v}^{t}\beta(X_v(s))(\sm_{D_v}(X_v(s)) - 1)\d s}  \notag\\
%& \times\prod_{v \in S_k(t) \backslash \cN_t}\frac{\langle \varphi, \cZ_v\rangle_{D_v}(X_v(\tau_v^-))}{\sm_{D_v}(X_v(\tau_v^-))} \notag\\
%&\times 
%\prod_{v \in S_k(t) \backslash\cN_t}\frac{N_v^{D_v}}{\varphi(X_v(\tau_v^-))^{M_v} \, \langle \varphi, \cZ_v\rangle_{D_v}(X_v(\tau_v^-))} \prod_{v \in S_k(t) \backslash \{\emptyset\}}  \varphi(X_v(\sigma_v))\notag \\
= &\prod_{v \in S_k(t)}\frac{\zeta(X_v, \tau_v^- \wedge t)}{\zeta(X_v, \sigma_v)}  \notag \\
& \times \prod_{v \in S_k(t) \backslash \cN_t}{\sm_{D_v}(X_v(\tau_v^-))} {\rm e}^{-\int_{\sigma_v}^{\tau_v}\beta(X_v(s))(\sm_{D_v}(X_v(s)) - 1)\d s}
%{\varphi(X_v(\tau_v^-))^{D(v)}}
\prod_{v \in \cN_t}{\rm e}^{-\int_{\sigma_v}^{t}\beta(X_v(s))(\sm_{D_v}(X_v(s)) - 1)\d s}  \notag\\
& \times\prod_{v \in S_k(t) \backslash \cN_t}\frac{\langle \varphi, \cZ_v\rangle_{D_v}(X_v(\tau_v^-))}{\sm_{D_v}(X_v(\tau_v^-))} \notag\\
&\times 
\prod_{v \in S_k(t) \backslash \{\emptyset\}} \frac{N_{p_v}^{D_{p_v}}}{\langle \varphi, \cZ_{p_v}\rangle_{D_{p_v}}(X_v(\tau_{p_v}^-))} \frac{\varphi(X_v(\sigma_v))}{\varphi(X_v(\tau_{p_v}^-))^{M_{p_v}}},\notag
\end{align}
where $p_v$ is the label of the parent of particle with label $v$. The above decomposition holds since 
\begin{align*}
  \prod_{v \in S_k(t) \backslash\cN_t}&\frac{N_v^{D_v}}{\varphi(X_v(\tau_v^-))^{M_v} \, \langle \varphi, \cZ_v\rangle_{D_v}(X_v(\tau_v^-))} \prod_{v \in S_k(t) \backslash \{\emptyset\}}  \varphi(X_v(\sigma_v))\\
  &= \prod_{v \in S_k(t) \backslash \{\emptyset\}} \frac{N_{p_v}^{D_{p_v}}}{\langle \varphi, \cZ_{p_v}\rangle_{D_{p_v}}(X_v(\tau_{p_v}^-))} \frac{\varphi(X_v(\sigma_v))}{\varphi(X_v(\tau_{p_v}^-))^{M_{p_v}}},
\end{align*}
by noting that in the product on the left-hand side of the above equality, each of the terms in the first product is a parent of an element in the second product. 
Using Remark \ref{rem:alt-description}, one can see that each of the terms above describes a change of measure with respect to $\mathbb{P}^k$ for, in order: the motion; branch rate; offspring distribution; and selection of the spine particles immediately after a branching event. These changes of measure account for the differences between the pathwise constructions of the measures $\mathbb{P}^k$ and $\mathbb{Q}^k$. Indeed, we have the following result. 

\end{rem}

\begin{prop}\label{lem-mg}
For $x \in E$, $({W}_t^k, t \ge 0)$ is a martingale. Define
\begin{equation}
\frac{\d \widetilde{\mathbb Q}_{\delta_x}^k}{\d{\mathbb P}_{\delta_x}^k}\Bigg|_{\mathcal{F}_t^k} = {W}_t^k,  \qquad t \ge 0, \, x \in E.
\label{Qtilde}
\end{equation}
Then 
% and $\widetilde{\mathbb{Q}}_{\delta_x}^k$ has the branching property
for all $x \in E$, $\mathbb{Q}_{\delta_x}^k= \widetilde{\mathbb{Q}}_{\delta_x}^k$.
\end{prop}

\begin{proof}
 In the spirit of \cite{CR}, and the proofs of Proposition 11 and Theorem 12 in \cite{HHK}, %(which encapsulates a different category of result), 
	it suffices to demonstrate that the   change of measure holds for the behaviour of the initial particle, up to and including its branching event. Thereafter, the Markov property ensures that the result is true in general.
\medskip

To this end, let us suppose that $T_1$ is the first branch time.  According to the definition of $\mathbb{Q}_{\delta_x}^k$ we have, for $x\in E$ and any bounded measurable $H$,
\[
\mathbb{Q}_{\delta_x}^k[H(\xi_s, 0\leq s\leq t); t< T_1] = 
\mathbf{E}_{x}\left[\frac{\zeta(\xi, t)}{\zeta(\xi, 0)}H(\xi_s, s\leq t) {\rm e}^{-\int_0^t \beta(\xi_s){\sm}_{k}(\xi_s)\d s}\right], \qquad t\geq0,
\]
where, on $\{t<T_1\}$, $\xi = \xi^1=\dots = \xi^k$. Noting similarly that 
\[
\mathbb{P}_{\delta_x}^k[H(\xi_s, 0\leq s\leq t); t< T_1] = 
\mathbf{E}_{x}\left[H(\xi_s, s\leq t) {\rm e}^{-\int_0^t \beta(\xi_s)\d s}\right], \qquad t\geq0,
\]
it follows that 
\[
\mathbb{Q}_{\delta_x}^k[H(\xi_s, 0\leq s\leq t); t< T_1] = 
\mathbb{P}^k_{x}\left[\frac{\zeta(\xi, t)}{\zeta(\xi, 0)}H(\xi_s, s\leq t) {\rm e}^{-\int_0^t \beta(\xi_s)({\sm}_{k}(\xi_s)-1)\d s}\right], \qquad t\geq0,
\]
which agrees with \eqref{Qtilde} on $\{t<T_1\}$.
\medskip

Next, we extend this to include what happens at the first branch event. Again referring to the definition of $\mathbb{Q}^k_{\delta_x}$, for $t>0$, $x\in E$, $n \in \mathbb{N}$ and $H$ and $\xi$ as before, $f\in B(E)$  and $i_1,\dots, i_k\in \mathbb{N}$ such that $|\{i_1,\dots, i_k\}| = n$, $L_{\mathbf{i}}:=\{l\in \mathbb{N} \text{ s.t. } i_j=l \text{ for some } 1\le j\le k\}$ 
%with $k_l:=|{j: i_j=l}|$ for $1\le l \le N$
\begin{align*}
&\mathbb{Q}_{\delta_x}^k[H(\xi_s, 0\leq s\leq t){\rm e}^{-\langle f, \mathcal{Z}\rangle}
\mathbf{1}_{\{ T_1\in \d t\}}\mathbf{1}_{\{i_1, \dots, i_k\in\{1,\dots N\}\}}\mathbf{1}_{\{  \psi_t^1 =i_1, \dots, \psi^k_t  =i_k \}}]\notag\\
&=\mathbb{P}^k_{x}\Bigg[\frac{\zeta(\xi, t)}{\zeta(\xi, 0)}H(\xi_s, s\leq t){\rm e}^{-\langle f, \mathcal{Z}\rangle}{\rm e}^{-\int_0^t \beta(\xi_s)({\sm}_{k}(\xi_s)-1)\d s}
\notag\\
& \hspace{1cm}
 \times\mathbf{1}_{\{i_1, \dots, i_k\in\{1,\dots N\}\}}\mathbf{1}_{\{  \psi_t^1 =i_1, \dots, \psi^k_t  =i_k \}}
\beta(\xi_t) \sm_k(\xi_t) \frac{{\sm}_{k,n}(\xi_t)}{\sm_{k}(\xi_t)} \frac{\langle \varphi, \cZ\rangle_{k, n}}{\mathcal{E}_{\xi_t}(\langle \varphi, \cZ\rangle_{k, n})}
\frac{\prod_{l\in L_{\mathbf{i}}} \varphi(x_l)}{\langle \varphi, \cZ\rangle_{k, n}}
\frac{1}{1/N^k}\Bigg]\d t,
 \notag\\
 &=\mathbb{P}^k_{x}\Bigg[\frac{\zeta(\xi, t)}{\zeta(\xi, 0)}H(\xi_s, s\leq t){\rm e}^{-\langle f, \mathcal{Z}\rangle}{\rm e}^{-\int_0^t \beta(\xi_s)({\sm}_{k}(\xi_s)-1)\d s}
 \notag\\
&\hspace{4cm}  
\times\mathbf{1}_{\{i_1, \dots, i_k\in\{1,\dots N\}\}}\mathbf{1}_{\{  \psi_t^1 =i_1, \dots, \psi^k_t  =i_k \}}
  \beta(\xi_t)\varphi(\xi_t)^{-n}
\prod_{l\in L_{\mathbf{i}}} \varphi(x_l) \,
N^k\Bigg]\d t,
\end{align*}
where the factor $1/N^k$ is removing the selection bias of the $k$-spines under $\mathbb{P}^k$. Noting that, in the above calculation, on $\{T_1 \in \d t\}$, $S_k(t) \backslash \mathcal{N}_t = \{\emptyset\}$, the set $S_k(t) \backslash \{\emptyset\},$  agrees with the offspring of $\emptyset$, $\varphi(X_\emptyset(\tau_\emptyset^-)) = \varphi(\xi_t)$, $M_\emptyset = n$ and $N_\emptyset = k$.  As such we note that on $\{T_1\in \d t\}$, the change of measure \eqref{Qtilde} is valid.
\end{proof}

\subsection{Spines at different times}\label{SS:differenttimes}

We would also like to consider the ``skeleton at different times''. %That is, for any $0 \le s_k \le \ldots \le s_1$, we would like to consider the process generated by $\{\xi_s^i, s \le s_i, i = 1, \dots, k\}$. 
To this end, fix $k \ge 1$, suppose $0 \le s_k \le \dots \le s_1$ and write $\us = (s_1, \dots, s_k)$. {Let $$\cN_{\mathbf{s}}:=\{(v_1,\dots, v_k): v_i\in \cN_{s_i} \, 1\le i \le k \}$$ and for $\mathbf{v}=(v_1,\dots, v_k) \in \cN_{\mathbf{s}}$, let $$S_k(\mathbf{v},\mathbf{s}) := \{w \in \Omega :  w \preceq v_i \text{ for some } 1 \le i \le k\}$$ be the ``skeleton'' formed by the ancestors of $v_i$ up to times $s_i$.} We also write {$$S_k(\mathbf{s}):=S_k((\psi_{s_1}^1,\dots, \psi_{s_k}^k), \mathbf{s})$$} for the skeleton generated by the spines and $$\partial S_k(\mathbf{s})=\{v\in S_k(\mathbf{s}) \text{ such that } \nexists\, w\in S_k(\mathbf{s}) \text{ with } v\prec w\}$$
for the ``leaves'' of this skeleton.  If for each $1\le i \le k$  we set \[
r_i = \sup\{s > 0 : \psi_s^i = \psi_s^k \text{ for some } k \text{ with } s_k \ge s\}.
\]
then for each $i$ with $r_i=s_i$, $v=\psi_{s_i}^i$ is an element of $\partial S_k(\us)$. We define a further set of labels associated to those $i$ for which $r_i>s_i$:
\[
  \hat\partial S_k(\us) = \{\psi_{r_i}^i, \, 1 \le i \le k, r_i > s_i\}.
\]
\begin{figure}
\centering
	\includegraphics[width = .4\textwidth]{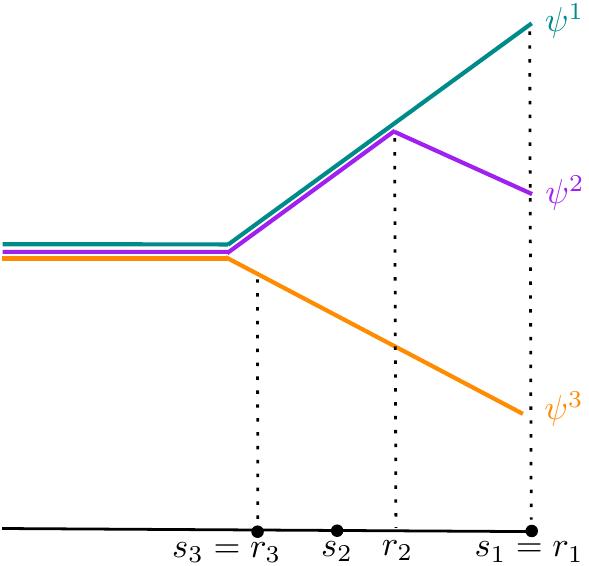}
	\caption{The spines carrying marks $1$, $2$ and $3$ are depicted in cyan, purple and orange respectively. The times $s_i$ and $r_i$ for $i=1,2,3$ are shown (the dotted vertical lines correspond to $r_1, r_2$ and $r_3$).} \label{fig:stubbies}
\end{figure}
\noindent Note that while $\partial S_k(\us)$ is a subset of $S_k(\us)$, $\hat\partial S_k(\us)$ is not.
Finally, for $v=\psi_{r_i}^i \in \hat \partial S_k(\us)\cup \partial S_k(\us)$ we define \begin{equation} \label{eq:rv} r_v:=r_i\end{equation}
(so for $v\in \psi_{s_i}^i \in \partial S_k(\us)$, we equivalently have $r_v=s_i$). See Figure \ref{fig:stubbies}.

\medskip

{ Now fix $t \ge s_1$ and let $\cF_{t, \us}^k$ denote the $\sigma$-algebra generated by:
\begin{itemize}\itemsep0em
	\item $\{\xi^i_s : s \le s_i, 1 \le i \le k\}$ (the motion of the spine with mark $i$ up to time $s_i$ for each $i$);
	\item $\{\psi^i_s : s \le s_i, 1 \le i \le k\}$ (the Ulam-Harris labels associated to the the spine with mark $i$ up to time $s_i$ for each $i$); 
		\item $\{\mathbf{b}_{\psi^i_s} : s \le s_i, 1 \le i \le k\}$ (the collection of marks carried by the spine with mark $i$ up to time $s_i$ for each $i$); 
	\item the subtree rooted at each $w \in \Omega$ that does not carry any marks and is a sibling of \emph{some} $v \in S_k(\mathbf{s})$\footnote{where by subtree we mean the subprocess started at time $\sigma_w$ with root label $w$}, considered up until (global) time $t$.
\end{itemize}  Note that the collection of random variables in the third bullet point above will not always be measurable with  respect to the collection in the second. For example, if $k=2$ and the spine carrying mark $1$ also carries mark $2$ at time $s\in (s_2,s_1)$, then since $\{\psi_s^i : s\le s_i, i=1,2\}$ does not tell us about the labels associated to the spine with mark $2$ after time $s_2$,  $\mathbf{b}_{\psi_s^1}$ is not measurable with respect to it.}

\medskip

We will also use the notation $$\cF_{\us}^k := \cF_{s_1, \us}^k.$$

\begin{figure}\label{fig:tree}
	\hspace{.75cm}
\includegraphics[width = 13cm, height = 8cm]{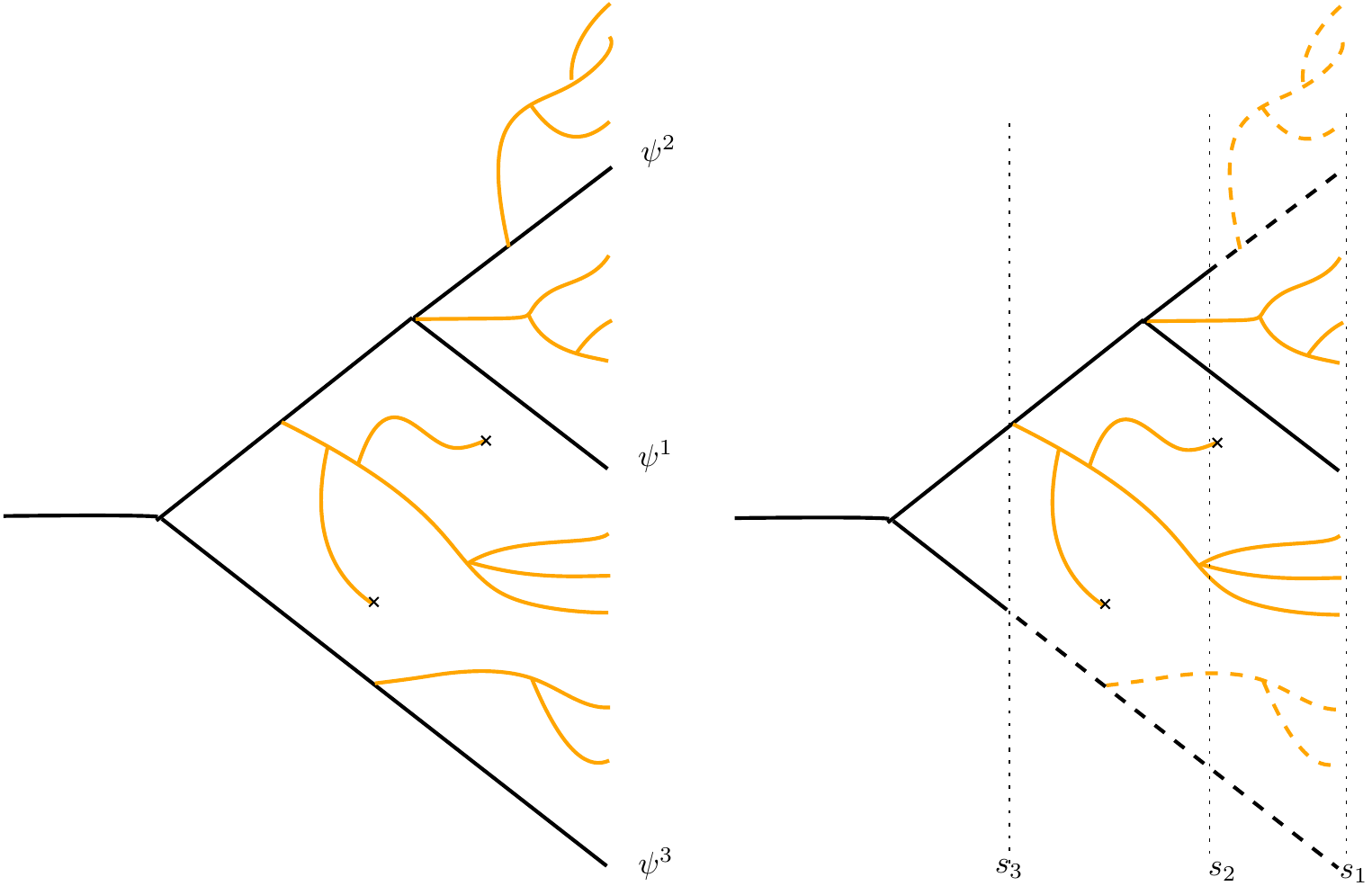}
\caption{The left-hand side figure shows the tree up to time $s_1$ with three spines marked in black. The right-hand side shows the information from the filtration $\cF_{s_1, \bf{s}}^3$, with ${\bf s} = (s_1, s_2, s_3)$, with dashed lines denoting information that is not included.}
\end{figure}

\noindent For each $\mathbf{s}$, we define an $\cF_{\us}^k$-measurable random variable
\begin{align}
{W}^k_{\mathbf{s}}
:= & %\mathbf{1}_{\{\zeta(X_v,s_v)>0\, \forall v\in \partial S_k(\mathbf{s})\}} 
\prod_{v \in S_k(\us)}\frac{\zeta(X_v, \tau_v^- \wedge r_v)}{\zeta(X_v, \sigma_v)} 
{\rm e}^{-\int_{\sigma_v}^{\tau_v \wedge r_v}\beta(X_v(s))({\sm}_{D_v}(X_v(s)) - 1)\d u}\notag \\
&\times \prod_{v \in S_k(\us) \backslash \{\emptyset\}} {\varphi(X_v(\sigma_v))} 
\prod_{v \in S_k(\us) \backslash\partial S_k(\us)}\frac{N_v^{D_v}}{\varphi(X_v(\tau_v^-))^{M_v}},
\label{gen-mg}
\end{align}
where we have used the notation $r_v$ defined in \eqref{eq:rv} for $v\in \partial S_k(\us)$ (and set $r_v=\infty$ otherwise).

%With this notation in hand, we define the $\cF_\us^k$-adapted process $(\widetilde{W}^k_{\mathbf{s}} : 0 \le s_k \le \dots \le s_1)$ as
%\[
%  \widetilde{W}^k_{\mathbf{s}} = h_\us(\psi_1, \dots, \psi_k).
%\]
%Note that we may write
%\[
% \widetilde{W}^k_{\mathbf{s}} = \sum_{\substack{v_i \in \cN_{s_i} \\ i = 1, \dots, k}}h_\us(v_1, \dots, v_k) \mathbf{1}_{\{\psi_{s_i}^i = v_i, \, i = 1, \dots, k\}}.
%\]
\medskip

Recalling Lemma \ref{lem-mg}, restricting instead to $\mathcal F_{\bf s}^k$ yields the following result.

\begin{lemma}
For each $k \ge 1$, $0 \le s_k \le \ldots \le s_1$ and $x \in E$, we have
\begin{equation}
\frac{\d \mathbb{Q}^k_{\delta_x}}{\d \mathbb{P}^k_{\delta_x}} \Bigg|_{\mathcal{F}^k_{\us}} = {W}^k_{\mathbf{s}}.
\label{COM-k}
\end{equation}
\end{lemma}

%\begin{rem}
%Note that the description of the process under $\mathbb{Q}^k$ given in the previous section is still valid. 
%\end{rem}

\begin{proof}

Fix $k \ge 1$, $0 \le s_k \le \dots \le s_1 \le t$ and $x \in E$. Then, due to the structure of ${W}_t^k$, we may write
\begin{align}
  {W}_t^k = {W}_{\bf s}^k\, &\times  
  \prod_{v \in S_k(t)\backslash S_k(\us)}\frac{\zeta(X_v, \tau_v^- \wedge t)}{\zeta(X_v, \sigma_v)} 
 {\rm e}^{-\int_{\sigma_v}^{\tau_v \wedge t}\beta(X_v(s))({\sm}_{D_v}(X_v(s)) - 1)\d u}\notag \\
 &\times \prod_{v \in \partial S_k(\us)}\frac{\zeta(X_v, \tau_v^- \wedge t)}{\zeta(X_v, r_v)}{\rm e}^{-\int_{r_v}^{\tau_v \wedge t}\beta(X_v(s))({\sm}_{D_v}(X_v(s)) - 1)\d u} \notag \\
 &\times \prod_{v \in V_{\us, t}}
 \frac{N_v^{D_v}}{\varphi(X_v(\tau_v^-))^{M_v}}
 \prod_{v \in S_k(t)\backslash S_k(\us)}\varphi(X_v(\sigma_v)),
\end{align}
where $V_{\us, t} := 
%(S_k(\us) \cap \mathcal N_{\color{red}\us}) \cup (S_k(t) \cap S_k(\us)^c \cap \mathcal{N}_t^c)$ {\color{red}[AEK: I think this should say $
%\partial S_k(\us) \cup (S_k(t) \cap S_k(\us)^c \cap \mathcal{N}_t^c)$ (that is, $V_{\mathbf{s},t}$ is  $
(S_k(t)\setminus S_k(\mathbf{s})\cup \cN_t) \cup \partial S_k(\mathbf{s}) $ (in words, the difference of the skeletons at times $t$ and $\mathbf{s}$, minus the boundary at time $t$, but plus the boundary at time $\mathbf{s}$).
%the set $\mathcal{N}_\us$ is a vector of labels so cannot be intersected with individual labels. Can this strange set be characterised in words?]}. 
Hence, 
\begin{align} \label{eq:mgale_project}
  \mathbb{P}_{\delta_x}^k[{W}_t^k | \cF_{\us}^k]
 = {W}_{\bf s}^k \, &\times \mathbb{P}^k_{\delta_x}\bigg[\prod_{ v \in S_k(t)\backslash S_k(\us)}\frac{\zeta(X_v, \tau_v^- \wedge t)}{\zeta(X_v, \sigma_v)} 
 {\rm e}^{-\int_{\sigma_v}^{\tau_v \wedge t}\beta(X_v(s))({\sm}_{D_v}(X_v(s)) - 1)\d u}\notag\\
  &\qquad \quad\times \prod_{v \in \partial S_k(\us)}\frac{\zeta(X_v, \tau_v^- \wedge t)}{\zeta(X_v, r_v)}{\rm e}^{-\int_{r_v}^{\tau_v \wedge t}\beta(X_v(s))({\sm}_{D_v}(X_v(s)) - 1)\d u} \notag \\
 &\qquad \quad \times \prod_{v \in V_{\us, t}}
 \frac{N_v^{D_v}}{\varphi(X_v(\tau_v^-))^{M_v}}
 \prod_{v \in S_k(t)\backslash S_k(\us)}\varphi(X_v(\sigma_v)) \bigg| \mathcal{F}_{\us}^k\bigg].
\end{align}

Now consider the collection of subprocesses initiated at times $r_v$ for each $v \in \partial S_k(\us)\cup  \hat\partial S_k(\us)$. The branching Markov property implies that these are 
%Note that the subtrees initiated at time $s_v$ from each $v \in \partial S_k(\us)$ are 
independent of each other and of $\mathcal{F}_{\us}^k$. 
Moreover, we may rewrite the right-hand side of \eqref{eq:mgale_project} as 
$$ W_{\mathbf{s}}^k \times \mathbb{P}^k_{\delta_x} \bigg[ \prod_{v\in \hat\partial S_k(\us) \cup \partial S_k(\us)} W_{t_v}^{(v)} \, \bigg| \, \cF_{\mathbf{s}}^k \bigg],$$
 where for each $v\in \partial S_k(\us)\cup \hat \partial S_k(\us)$, $W^{(v)}$ is a copy of the martingale $W^{D_v}$ associated to the subprocess rooted at $v$ and $t_v=t-r_v$. In particular, $$\mathbb{P}^k_{\delta_x} \bigg[  \prod_{v\in \hat\partial S_k(\us)\cup\partial S_k(\us)} W_{t_v}^{(v)} \, \bigg| \, \cF_\mathbf{s}^k \bigg] = \prod_{v\in \hat\partial S_k(\us) \cup\partial S_k(\us)} \mathbb{P}^k_{\delta_x} [W_{t_v}^{(v)}] =1,$$
%Using the decomposition in Remark \ref{rem:alt-mg} and similar arguments to those given in the proof of Proposition \ref{lem-mg} show that the conditional expectation on the right-hand side above is equal to one, 
which gives the result.
\end{proof}

%\bigskip
%

\section{Many-to-few lemma}\label{S:m2few}

\subsection{Statement and proof of the many-to-few lemma}\label{SS:m2few}
We are now ready to formulate and prove our main result, which is a many-to-few lemma for general non-local branching Markov processes at a collection of different times. This generalises the result of \cite{Many2few} in two ways: firstly, it holds for non-local branching mechanisms; and secondly, it allows us to deal with sums over the population at different times.

\begin{lemma}[Many-to-few at different times] \label{m2k-ktimes}
{Let $x \in E$, $k \ge 1$ and $0 \le s_k \le \dots \le s_1$ be fixed. 
Suppose that 
\[ 
Y= \sum_{v_i \in \cN_{s_i},\, i = 1, \dots, k} Y(v_1, \dots, v_k)\mathbf{1}_{\{\psi_{s_i}^i=v_i,\, 1\le i \le k\}}
%{\color{red} Y(\psi_{s_1}^1, \dots, \psi^k_{s_k}) }
\]
is non-negative and $\mathcal{F}_{\mathbf{s}}^k$-measurable, where $Y(v_1,\dots, v_k)$ is $\mathcal{F}_{s_1}$-measurable\footnote{By this we mean that we have a collection $(Y(v_1, \dots, v_k), v_1, \dots, v_k \in \Omega)$ of $\cF_{s_1}$-measurable random variables such that $Y(v_1, \dots, v_k) = 0$ unless $v_i \in \cN_{s_i}$ for all $i = 1, \dots, k$.} for every $(v_1,\dots, v_k)\in \mathcal{N}_{\mathbf{s}}$. Then it holds that}
\begin{align*}
%\mathbb{P}^k_{\delta_x}[Y] 
%&=
\mathbb{P}_{\delta_x}\Bigg[
\sum_{\substack{v_i \in \cN_{s_i} \\ i = 1, \dots, k}} Y(v_1, \dots, v_k)
%= \mathbf{1}_{\{\psi_{s_i}^i=v_i\, 1\le i \le k\}}
\Bigg] 
 %\notag\\
  &= \mathbb{Q}_{\delta_x}^{{k}}\bigg[{Y}
 \prod_{v \in S_k(\us)}\frac{\zeta(X_v, \sigma_v)}{\zeta(X_v, \tau_v^- \wedge s_v)} 
{\rm e}^{\int_{\sigma_v}^{\tau_v \wedge s_v}\beta(X_v(s))({\sm}_{D_v}(X_v(s)) - 1)\d u}\notag \\
&\hspace{0.8cm} \times\frac{\prod_{v \in S_k(\us) \backslash  \partial S_k(\us)}\varphi(X_v(\tau_v^-))^{M_v}}
{\prod_{v \in S_k(\us) \backslash \{\emptyset\}}\varphi(X_v(\sigma_v))}
\prod_{v \in S_k(\us)\backslash \partial S_k(\us)}N_v^{{-D_v}}\prod_{i= 1}^k \prod_{\emptyset \preceq v \prec \psi_{s_i}^i}N_v
 \bigg].
 \label{m2few-general}
\end{align*}
\end{lemma}
{ \begin{rem}\rm \label{notequal}
Recall that 
%From its definition, $S_k(\mathbf{s}) = \{w \in \Omega : w \preceq \psi^i_{s_i} \text{ for some } 1 \le i \le k\}$ and $ \cN_{\mathbf{s}}:=\{(v_1,\dots, v_k): v_i\in \cN_{s_i} \, 1\le i \le k \}$. Hence 
\[
S_k(\us) \backslash  \partial S_k(\us)  = \{w \in \Omega : \emptyset \preceq w \prec \psi^i_{s_i} \text{ for some } 1 \le i \le k\}.
\]
%In particular 
%\[
%\prod_{v\in S_k(\us) \backslash  \cN_{\us} } \cdot = \prod_{\substack{\emptyset \preceq v \prec \psi_{s_i}^i \\ i = 1,\dots, k}}\cdot
%\]
Also recall that the mark carried by an individual (spine) $v$ is ${\bf b}_v$ with  cardinality $D_v=|{\bf b}_v|$ (that is the number of marks carried by $v$) and $M_v$ is the number of offspring of $v$ that inherit a mark from $v$. As a consequence, 
\begin{equation}
|\{i\in\{1,\dots, k\}: \emptyset \preceq v \prec \psi_{s_i}^i\}| \leq  D_v.
\label{problem}
\end{equation}
The inequality in \eqref{problem}  can be strict when, for example,   $\emptyset\preceq v  = \psi_{s_j}^j \prec \psi_{s_i}^i$ for some pair $i\neq j$.
%Similarly, it seems that the notion of a separated skeleton equally forbids $v = \psi_{s_j}^j \prec \psi_{s_i}^i$.
\end{rem}}

\begin{proof}[Proof of Lemma \ref{m2k-ktimes}]
Let us start by rewriting the right-hand side of the expression in the lemma. Noting that under $\mathbb{Q}_{\delta_x}$, $W_{\us}^k$ is positive, we have
\begin{align} 
\mathbb{Q}_{\delta_x}^k\bigg[ \frac{Y}{{W}^k_{\mathbf{s}}} \prod_{i= 1}^k \prod_{\emptyset \preceq v \prec \psi_{s_i}^i}N_v \bigg]
&= \mathbb{P}_{\delta_x}^k \bigg[ Y\prod_{i= 1}^k \prod_{\emptyset \preceq v \prec \psi_{s_i}^i}N_v\bigg]\notag\\
&= \mathbb{P}^k_{\delta_x} \bigg[ \sum_{u_i\in \cN_{s_i}, \,  1\le i \le k}Y(u_1,\dots, u_k)\mathbf{1}_{\{\psi_{s_i}^i=u_i \,  1\le i \le k \}} \prod_{i= 1}^k \prod_{ \emptyset \preceq v \prec u_i}N_v\bigg] \label{lookslike}
\end{align}
where the first equality holds thanks to the change of measure~\eqref{COM-k}. %Note that for any $v_1,\dots, v_k\in \cN_{\mathbf{s}_1}$ we can write

\medskip

%Next, for $u_1,\dots, u_k\in \cN_{s_1}$, write $\pi_{s_i}( u_i)$ for the unique ancestor of $u_i$ at time $s_i$. That is to say, $\psi^i_{s_i}=\pi_{s_i}(\psi^i_{s_1})$. 
%Then by expressing the partitioning summation in the definition of $Y$ as a partition over the choices of the spines via their the identification  in  $\cN_{s_1}$, the right hand side above is equal to  
%\begin{align}
%%\mathbb{P}_{\delta_x}^k& \bigg[ Y(\pi_{s_1}(\psi^i_{s_1}),\dots, \pi_{s_k}(\psi^i_{s_1})) \prod_{i= 1}^k \prod_{ \emptyset \preceq v \prec \pi_{s_i}\circ\psi^i_{s_1}}N_v\bigg]\notag\\
%\mathbb{P}^k_{\delta_x} \bigg[ \sum_{u_i\in \cN_{s_1}, \,  1\le i \le k}Y(\pi_{s_1}(u_1),\dots, \pi_{s_k}(u_k))\mathbf{1}_{\{\psi_{s_1}^i=u_i \,  1\le i \le k \}} \prod_{i= 1}^k \prod_{ \emptyset \preceq v \prec \pi_{s_i}(u_i)}N_v\bigg]
%\label{lookslike}
%\end{align}
Conditioning on {$\mathcal{F}_{s_1}$, we have}
%Conditioning on $\mathcal{F}_{s_1}$, and using that 
\[
  \mathbb{P}^k_{\delta_x}\left(\psi_{s_i}^i = u_i, \, i = 1, \dots, k \bigg| \mathcal{F}_{s_1}\right)
 = \prod_{i=1}^k \prod_{\emptyset \preceq v \prec u_i} N_v^{-1},
\]
and {thus, by the properties of conditional expectation,} we can rewrite \eqref{lookslike} as
\begin{align*} 
&\mathbb{P}^k_{\delta_x} \bigg[ \sum_{u_i\in \cN_{s_i}, \,  1\le i \le k}Y(u_1,\dots, u_k)
%\mathbf{1}_{\{\psi_{s_1}^i=u_i \,  1\le i \le k \}} 
  \prod_{i= 1}^k \prod_{\emptyset \preceq v \prec u_i}N_v \prod_{\emptyset \preceq v \prec u_i} N_v^{-1}\bigg] 
%& = \mathbb{P}^k_{\delta_x} \bigg[ \sum_{u_i\in \cN_{s_i}, \,  1\le i \le k}Y(u_1,\dots, u_k)
%%\mathbf{1}_{\{\psi_{s_1}^i=u_i \,  1\le i \le k \}} 
%\prod_{i=1}^k \prod_{v_i(u_i) \preceq v \prec u_i} N_v^{-1}  \bigg] \\
 = \mathbb{P}_{\delta_x} \bigg[ 
\sum_{v_i\in \cN_{s_i}, \,  1\le i \le k}
Y(v_1,\dots, v_k) 
%\mathbf{1}_{\{\psi_{s_i}^i=v_i\, 1\le i \le k\}}
 \bigg],
% \notag\\
%&=\mathbb{P}^k_{\delta_x}[Y].
\end{align*}
as required.
\end{proof}

Referring to Remark \ref{notequal}, if we take $(s_1,\dots, s_k) = (t, \dots, t)$ then 
$v = \psi_{t}^j \prec \psi_{t}^i$ cannot occur and the inequality in \eqref{problem} is an equality. 
In that case, 
\begin{align*}
\prod_{i= 1}^k \prod_{\emptyset \preceq v \prec \psi_{s_i}^i}N_v &= \prod_{v\in S_k(\us) \backslash  \partial S_k(\us) }N_v^{D_v},
\end{align*}
%Note that if $\mathbf{s}=(t,t,\dots, t)$ we have that 
%\[
%\prod_{v \in S_k(\us)\backslash \partial S_k(\us)}N_v^{{D_v}}=\prod_{i= 1}^k \prod_{\emptyset \le v < \psi_{s_i}^i}N_v,
%\]
%
which yields the following corollary (as a special case).

\begin{cor}[Many-to-few]\label{lem-many2few}
Let $x \in E$, $k \ge 1$ and $t \ge 0$ be fixed. 
Suppose that 
\[ 
Y= \sum_{v_i \in \cN_{t},\, i = 1, \dots, k} Y(v_1, \dots, v_k)\mathbf{1}_{\{\psi_{t}^i=v_i,\, 1\le i \le k\}}
%{\color{red} Y(\psi_{s_1}^1, \dots, \psi^k_{s_k}) }
\]
is $\mathcal{F}_{t}^k$-measurable with $Y(v_1,\dots, v_k)$ $\mathcal{F}_{t}$-measurable for $v_i\in \mathcal{N}_{t}$ for $i = 1, \dots k$. 
%For any $k \ge 1$, $t \ge 0$, $x \in E$ and $\cF_t^k$-measurable $Y$, so that $$Y=\sum_{\substack{v_i \in \cN_{t} \\ i = 1, \dots, k}} Y(v_1,\dots, v_k) \mathbf{1}_{\{\psi_{s_i}^i=v_i\, 1\le i \le k\}} $$
%with $Y(v_1,\dots, v_k) \; \mathcal{F}_t$ measurable for each $v_1,\dots, v_k$, 
Then
\begin{align*}
\mathbb{P}_{\delta_x}\bigg[
\sum_{\substack{v_i \in \cN_{t} \\ i = 1, \dots, k}} Y(v_1, \dots, v_k)
% \mathbf{1}_{\{\psi_{s_i}^i=v_i\, 1\le i \le k\}}
\bigg]   = \mathbb{Q}_{\delta_x}^k&\bigg[Y \prod_{v \in S_k(t)}
\frac{\zeta(X_v, \sigma_v)} {\zeta(X_v, \tau_v^- \wedge t)}{\rm e}^{\int_{\sigma_v}^{\tau_v \wedge t}\beta(X_v(s))({\sm}_{D_v}(X_v(s)) - 1)\d s}\\
& \hspace{2cm}\times \frac{\prod_{v \in S_k(t) \backslash \mathcal{N}_t}{\varphi(X_v(\tau_v^-))^{M_v}}}{\prod_{v \in S_k(t) \backslash \{\emptyset\}} \varphi(X_v(\sigma_v))}\bigg]. 
\end{align*}
\end{cor}

\subsection{{Natural} choice of $\zeta$ and separated skeletons}\label{sec:sep-skel}
We will now consider a specific example of the many-to-few formula given in Lemma \ref{m2k-ktimes} by choosing a particular form for the martingale $\zeta$. 
Using the fact that $\varphi$ is the right eigenfunction for the MBP, $(X, \mathbb{P})$, with corresponding eigenvalue $\lambda$, it isn't too difficult to show that under $\mathbf{P}_{x}$
\begin{equation}
\zeta(\xi, t) = \frac{\varphi(\xi_t)}{\varphi(x)}\exp\left(\int_0^t {\beta}(\xi_s) (\sm_1(\xi_s)-1)  \d s\right), \quad t \ge 0,
\label{specific-zeta}
\end{equation}
is a martingale.

\medskip

In this case, we have that 
\begin{align}
W^k_t&=\frac{1}{\varphi(x)} \prod_{v \in S_k(t)}{\rm e}^{-\int_{\sigma_v}^{\tau_v \wedge t}\beta(X_v(s))({\sm}_{D_v}(X_v(s)) - {\sm}_1(X_v(s)))\d s}\notag \\
&\qquad \qquad\times \prod_{v\in S_k(t)\cap \cN_t} \varphi(X_v(t)) \prod_{v \in S_k(t) \backslash \mathcal{N}_t}\frac{N_v^{D_v}}{\varphi(X_v(\tau_v^-))^{M_v-1}}.
\label{eq:special-W}
\end{align}

We now further assume that each of the nodes $v_i \in \mathcal{N}_{s_i}$, $i = 1, \dots, k$, that make up the skeleton $S_k(\uv, \us)$, are distinct. We refer to a skeleton with this property as a {\it separated skeleton}. This implies that at time $s_i$, node $v_i$ only carries one mark for each $i = 1, \dots, k$. {Recalling again Remark \ref{notequal},  we thus have the advantage of writing}
\[
  \prod_{v \in S_k(\us)\backslash \mathcal{N}_\us} N_v^{-D_v} 
  = \prod_{i = 1}^k \prod_{\emptyset \le v < \psi_{s_i}^i} N_v^{-1}.
\]

\medskip

Applying the many-to-few formula with this martingale, for the special choice of $\zeta$, and in the case of separated skeletons, we get 
\begin{align}\label{eq:natural_martingale}
&\mathbb{P}_{\delta_x}^k \bigg[ 
\sum_{v_i\in \cN_t \text{ distinct }} Y(v_1,\dots, v_k) 
%Y \mathbf{1}_{\{\{\psi_t^i\}_{1\le i \le k} \text{ distinct}\}}
\bigg] \notag \\
&= \varphi(x)\mathbb{Q}^k_{\delta x} \bigg[ Y \mathbf{1}_{\{\{\psi_t^i\}_{1\le i \le k} \text{ distinct}\}} \prod_{j=1}^k \varphi(\xi^j_t)^{-1} \notag \\
& \hspace{3cm}\times \prod_{v \in S_k(t)}{\rm e}^{\int_{\sigma_v}^{\tau_v \wedge t}\beta(X_v(s))({\sm}_{D_v}(X_v(s)) - {\sm}_1(X_v(s)))\d s}
\prod_{v \in S_k(t) \backslash \mathcal{N}_t}{\varphi(X_v(\tau_v^-))^{M_v-1}}\bigg].
\end{align}

\section{Application to genealogies in the critical case}\label{S:application}

As an application, we determine the asymptotic law of the death time of the most recent common ancestor, henceforth referred to as \emph{split time}, of two particles sampled uniformly from a critical population at two different times. The limit is taken as $t\to \infty$, when we have conditioned on survival of the process up to time $t$.

\medskip

We assume in this section that the measures $\mathbb{Q}^k$ and $\mathbb{P}^k$ are as defined in section \ref{sec:sep-skel}. We remind the reader of the notation $\xi^1,\dots, \xi^k$ for the motion of the $k$ spines under $\mathbb{Q}^k$ and $\mathbb{P}^k$, and $\psi^1,\dots, \psi^k$ for the Ulam-Harris labels that they carry.

\medskip

We assume throughout the section that our branching process satisfies the following additional criticality requirement.

\begin{ass}\label{A:Yaglom}
The following criticality assumptions hold.
 \begin{enumerate}
 \item $\lambda = 0$, where $\lambda$ was introduced in Assumption \ref{ass1}
 \item Define $\Delta_t := \sup_{x \in E, f \in B_1^+(E)}|\varphi(x)^{-1}\sT_t[f](x) -\langle f, \tilde\varphi\rangle |$.
 Then 
 \[
   \sup_{t \ge 0}\Delta_t < \infty \text{ and } \lim_{t \to \infty}\Delta_t = 0.
 \]
 \item The number of offspring produced at a branching event is bounded above by $n_{max} < \infty$.
 \item There exists a constant $C > 0$ such that for all $g \in B^+(E)$, 
 \[
 {\Sigma  :=} \langle \beta \mathbb{V}[g] ,  \tilde\varphi\rangle \ge C\langle g,  \tilde\varphi\rangle^2,
 \]
 where $\mathbb{V}[g](x) := \mathcal{E}_x[\langle g, \mathcal{Z}\rangle_{2, 2}]$ and where the notation $\langle g, \mathcal{Z}\rangle_{k, n}$ was defined in \eqref{eq:new_offspring}.
 \item For all $x \in E$, $\mathbb{P}_{\delta_x}(\exists t > 0 \text{ such that } N_t = 0) = 1$.
 \end{enumerate}
\end{ass}
\begin{rem}\rm{
 We note that the assumptions above are inherited from \cite{YaglomNTE, moments}, where asymptotic results concerning criticality and moment growth were considered. Whereas the Assumption \ref{A:Yaglom}.1 is clearly a standard criticality assumption, the remaining assumptions \ref{A:Yaglom}.2 - \ref{A:Yaglom}.5 can be interpreted as follows. Roughly speaking, Assumption \ref{A:Yaglom}.2 describes the uniform stability of the mean semigroup. In particular, by taking $f\equiv1$, Assumption \ref{A:Yaglom}.2 ensures that the first moment of the process settles down to a 
stationary value. Assumption \ref{A:Yaglom}.3 rather obviously requires the random number of offspring to be deterministically bounded.   Assumption \ref{A:Yaglom}.4 can be thought of as an irreducibility condition written in terms of the two-point correlation (or variance) functional  $\mathbb{V}[g]$ and ensures a minimal level of spatial mixing occurring for second order effects associated to the semigroup assumption in Assumption \ref{A:Yaglom}.2. Finally Assumption \ref{A:Yaglom}.5 ensures that, even though at criticality in Assumption \ref{A:Yaglom}.1, we can guarantee there is extinction almost surely. Assumption \ref{A:Yaglom}.5 is automatically satisfied e.g. for a branching Brownian motion in a compact domain with killing on the boundary, or the category of neutron branching process considered in \cite{SNTE-I, SNTE-II, SNTE-III}.}
\end{rem}

Let us now present our main application. Here and in the rest of this section $\Rightarrow$ means convergence in distribution.

\begin{prop}\label{P:split2}
Let $0 < a <1$ and let $x\in E$ be fixed. Let $T_t$ have the $\mathbb{P}_{\delta_x}(\cdot | N_t > 0)$ law of the split time of two particles: one chosen uniformly from those alive at time $t$ and one chosen uniformly from those alive at time $at$. Then $$\frac{T_t}{t}\Rightarrow T \quad \text{ as } t\to \infty$$  where $T$ has density 
	$$ f_a(u):=\frac{2a}{1-a} \frac{2(a-u)\log(1-\tfrac{u}{a})-(2-u-\tfrac{u}{a})\log(1-u)}{u^3}$$
	with respect to Lebesgue measure $\d u$ on $[0,a]$. 
\end{prop}

\begin{rem}\label{R:density1}\rm	Note that 
	\begin{align*} 
	\lim_{a \nearrow 1}\frac{2a}{1-a} &\frac{2(a-u)\log(1-\tfrac{u}{a})-(2-u-\tfrac{u}{a})\log(1-u)}{u^3} 
	= 2 (-2u+(u-2)\log(1-u))
	\end{align*}
	This agrees with the density in the case $a=1$ (for critical GW processes) calculated in \cite{GWgenealogy}.
\end{rem}

\begin{rem} It is not obvious a priori that $f_a$ is the density function of a random variable. However, this is indeed the case, since we can calculate that
	\begin{align*} 
		 \frac{2a}{1-a}& \int_0^a \frac{2(a-u)\log(1-\tfrac{u}{a})-(2-u-\tfrac{u}{a})\log(1-u)}{u^3} \, du \\
		= & \frac{2}{a(1-a)} \int_0^1 \frac{2a(1-u)\log(1-u)-(2-ua-u)\log(1-au)}{u^3} \, du \\
		= & \frac{2}{a(1-a)} \left[\frac{(1-u)}{u^2} \left(a(u-1)\log(1-u)+(1-au)\log(1-au)\right)\right]_0^1 \\
		= & -\frac{2}{a(1-a)}\lim_{u\to 0}\frac{a(u-1)(-u-\tfrac{u^2}{2}+\mathrm{o}(u^2)) + (1-au)(-au-\tfrac{a^2u^2}{2} + \mathrm{o}(u^2)) }{u^2} \notag\\
		= & 1.
	\end{align*}
\end{rem}

 We also have the following result concerning the joint convergence of the (normalised) population size at two different times under $\mathbb{Q}^1$. Recall from Assumption \ref{A:Yaglom}.3 that $\Sigma: = \langle \beta\mathbb{V}[\varphi], \tilde\varphi\rangle.$

\begin{prop}\label{prop:joint-conv}
Under $\mathbb{Q}_{\delta_x}^1$, 
$$\bigg(\frac{N_{at}}{t}, \frac{N_t}{t}\bigg) \Rightarrow (Z, \hat{Z}), \qquad \text{as } t \to \infty,$$ 
where $Z$ is equal in law to a $\mathtt{Gamma}(2, (a\Sigma\langle 1, \tilde\varphi\rangle/2)^{-1})$ random variable and, conditionally on $Z$, the law of $\hat{Z}$ is that of a $\mathtt{Gamma}(2+K, (\Sigma(1-a)\langle 1, \tilde\varphi\rangle /2)^{-1})$ random variable where $K \sim \mathtt{Poisson}(((1-a)\Sigma \langle 1, \tilde\varphi\rangle/2)^{-1}Z)$.

\medskip

Equivalently, under $\mathbb{P}_{\delta_x}(\cdot | N_t > 0)$, we have the joint convergence of $(N_{at}/t, N_t/t)$ to $(Y,\hat{Y})$, where the joint law of $(Y,\hat{Y})$ is that of $(Z,\hat{Z})$ \emph{weighted by} $1/\hat{Z}$.
\end{prop}

Before moving on to the proofs of the above two propositions, we first state a lemma that will be used throughout the aforementioned proofs. 

\begin{lemma}\label{L:Yaglom}
In what follows, we suppose that $(g_t, t\ge 0)$ are a collection of functions with $g_t\in B_1^+(E)$ for each $t>0$, and such that $g_t\to g\in B^+_1(E)$ pointwise as $t\to \infty$. For any $x\in E$, the following hold.
	\begin{enumerate}
		\item[(a)]  $t\mathbb{P}_{\delta_x}(N_t>0)\to \frac{2\varphi(x)}{\Sigma}$ as $t\to \infty$, $\sup_{t,x} |t\mathbb{P}_{\delta_x} (N_t>0)|<\infty$ and $\inf_{t} |t\mathbb{P}_{\delta_x} (N_t>0)| > 0$.
		\item[(b)]  The joint law of $$\left(\frac{X_t[g_t]}{t},\frac{N_t}{t}\right) \text{ under } \mathbb{P}_{\delta_x}(\cdot | \, N_t>0)$$ converges to that of $(\langle g, \tilde{\varphi} \rangle Z,\langle 1,\tilde{\varphi}\rangle Z)$ as $t\to \infty$, where $Z\sim \mathtt{Exponential}({2}/{\Sigma})$.
		\item[(c)] The joint law of 
		\begin{equation}
		\left(\frac{X_t[g_t]}{t},\frac{N_t}{t}, \xi_t\right) \text{ under } \mathbb{Q}^1_{\delta_x}
		\label{thetriple}
		\end{equation} converges to that of $(\langle g, \tilde{\varphi} \rangle Z, \langle 1,\tilde{\varphi}\rangle Z,\bar\xi)$ as $t\to \infty$, where $Z\sim \mathtt{Gamma}(2,2/\Sigma)$ and $\bar \xi$ is independent of $Z$, with law given by 
		\begin{equation}\label{eq:law-xi}
		P(\bar\xi \in A) = \langle \mathbf{1}_A\varphi, \tilde\varphi \rangle
		\end{equation}
		
		%having density $\varphi \tilde{\varphi}$ with respect to Lebesgue measure on $E$.
	\end{enumerate}
\end{lemma}

\begin{proof}
	(a) follows from \cite[Theorem 1.2]{YaglomNTE}, \cite[Lemma 7.4]{YaglomNTE} and \cite[Lemma 7.2]{YaglomNTE}. 
	For the proof of (b), first note that the joint convergence of $(X_t[g]/t, N_t/t)$ under $\mathbb{P}_{\delta_x}(\cdot | \, N_t>0)$ follows from \cite[Theorem 1.3]{YaglomNTE}, together with the fact that $t^{-1}[X_t[f]-\langle f,\tilde{\varphi}\rangle X_t[\varphi]]\to 0$ in probability as $t\to \infty$ for any bounded $f$ (see the proof of \cite[Theorem 1.3]{YaglomNTE}).
 	Writing 
	\[
	  X_t[g_t] = X_t[g] + X_t[g_t - g],
	\] 
	 it thus suffices to show that $X_t[g_t - g]/t \to 0$ in $\mathbb{P}_{\delta_x}(\cdot | N_t > 0)$-probability. For this, we will show that we actually have $L^1$ convergence. Setting $f_t = |g_t - g|$ (which is bounded by $2$), we have 
	\begin{align*}
	  \frac1t\mathbb{E}_{\delta_x}[|X_t[g_t - g]| | N_t > 0] 
	  &\le \frac{1}{t\mathbb{P}_{\delta_x}(N_t > 0)}\mathbb{E}_{\delta_x}[X_t[f_t]]\\
	  &\le  \frac{1}{t\mathbb{P}_{\delta_x}(N_t > 0)}(|\mathbb{E}_{\delta_x}[X_t[f_t]] - \varphi(x)\langle f_t, \tilde\varphi \rangle| + \varphi(x)\langle f_t, \tilde\varphi \rangle).\\
	\end{align*}
	From Lemma \ref{L:Yaglom}(a), it follows that $(t\mathbb{P}_{\delta_x}(N_t > 0))^{-1}$ is uniformly bounded. By the first part of Assumption \ref{A:Yaglom}, the first term in the parentheses on the right-hand side above converges to $0$ uniformly. Finally, by dominated convergence (since $\tilde\varphi$ is a finite measure), the second term in the parentheses also converges to $0$.

\medskip

		For (c), note that if $W_t^1$ is the martingale from \eqref{COM-k} (in the classical case of one spine),
		\begin{equation}
		  \mathbb{P}_{\delta_x}[W_t^1 | \cF_t] = X_t[\varphi],
		  \label{eq:mg}
		\end{equation}
		so that for any bounded, continuous function $F$,  
		\begin{align}
		  \mathbb{Q}_{\delta_x}^1\left[F\left(\frac{X_t[g_t]}{t}, \frac{N_t}{t}\right)\right] 
		  = \frac{t \mathbb{P}_{\delta_x}(N_t > 0)}{\varphi(x)} 
		  \mathbb{P}_{\delta_x}\left[F\left(\frac{X_t[g_t]}{t}, \frac{N_t}{t}\right)\frac{X_t[\varphi]}{t}{\,\bigg|\, N_t>0}\right].
		  \label{showsgamma}
		\end{align}
		Taking $F$ to be of the form $F(x, y) = {\rm e}^{-\theta x - \mu y}$, for $\theta, \mu \ge 0$ and using (a) and (b) yields 
		  the convergence of the first two components of the triple under $\mathbb{Q}_{\delta_x}^1$. The marginal convergence in law of $\xi_t$ to $\xi$ under $\mathbb{Q}_{\delta_x}$ follows from \cite[section 5]{YaglomNTE}.

	\medskip

		To see the joint convergence in law of the triple  in \eqref{thetriple}, note that due to the aforementioned marginal convergence of $\xi$ and the first two components, we immediately have tightness. Moreover, any subsequential limit has the form $({\langle g, \tilde\varphi\rangle}Z,\langle 1,\tilde{\varphi} \rangle Z,\bar\xi)$, where the marginals of $Z$ and $\bar\xi$ are as desired. Thus it remains to show that for any such subsequential limit, $Z$ and $\bar\xi$ (or equivalently $\langle 1,\tilde{\varphi} \rangle Z$ and $\bar\xi$) are independent.

\medskip

For this, define $N_t^*$ to be the contribution to $N_t$ of all descendants branching off the (single) spine particle \emph{before} time $t-t^{1/3}$. Then $N_t - N^*_t$ behaves like $N_{t^{1/3}}$ under $\mathbb{Q}^1$.  Applying the Markov property at time $t-t^{1/3}$, it follows that $t^{-1/3}({N_t - N_t^*})$ converges in law to  $\langle 1,\tilde{\varphi}\rangle Z$, where $Z\sim \mathtt{Gamma}(2,2/\Sigma)$, thanks to the discussion following \eqref{showsgamma}. Thus, for $\varepsilon > 0$, we have that $\mathbb{Q}_{\delta_x}^{1}(N_t - N_t^* \ge t^{1-\varepsilon}) \to 0$ as $t \to \infty$.  We also have that $\mathbb{Q}_{\delta_x}^{1}(N_t \ge t^{1-\varepsilon/2}) \to 0$, uniformly in $x$, thanks to the Markov inequality and Assumption \ref{A:Yaglom}.2.

{ Now note that,  on the one hand, $N_t^*/N_t\leq 1$. On the other hand, 
$1-(N^*_t/N_t) = (N_t-N^*_t)/N_t = t^{-1}(N_t-N^*_t)/(t^{-1}N_t )$. In the final equality, the numerator tends to zero in $\mathbb{Q}^1$-probability, and the denominator converges under   $\mathbb{Q}^1$ to a Gamma distributed random variable. It follows that  $N_t^*/N_t\to 1$ under $\mathbb{Q}_{\delta_x}^1$ as $t\to \infty$. Next note that the part of the spatial branching tree consisting of all descendants branching off the  spine particle \emph{before} time $t-t^{1/3}$ is conditionally independent (given the position of the spine at time $t-t^{1/3}$) of descendants branching off the spine particle \emph{after} time $t-t^{1/3}$.} This implies that any subsequential {distributional} limit  of $(N_t/t, \xi_t)$ as $t\to \infty$ under $\mathbb{Q}_{\delta_x}^1$, say $(\langle 1,\tilde{\varphi}\rangle Z,\bar\xi)$, can be extended to a subsequential limit $(Y^*,\langle 1,\tilde{\varphi}\rangle Z,\bar\xi)$ of $(N_t^*/t,N_t/t,\bar\xi)$ satisfying $Y^*=\langle 1,\tilde{\varphi}\rangle Z$ almost surely. On the other hand, ergodicity of the spine motion \cite[section 5]{YaglomNTE} implies that $Y^*$ and $\bar\xi$ are independent. That is to say, $Z$ and $\bar\xi$ are independent. 
%In terms of the subsequential limit $(Z,\langle 1,\tilde{\varphi}\rangle Z,\xi)$, this means that $\langle 1,\tilde{\varphi}\rangle Z$ and $\xi$ are independent, as required.
\end{proof}

\begin{rem}\rm
	\label{R:Yaglom}
Observe that a slight variant of the proof of (c) given above, implies that for any $c\in (0,1]$, the joint law of $(X_t[g_t]/t,N_t/t,\xi_{ct},\xi_t)$ under $\mathbb{Q}_{\delta_x}^1$ also converges to that of $(\langle g , \tilde{\varphi} \rangle Z,\langle 1,\tilde{\varphi} \rangle Z,\bar\xi, \bar\xi')$ as $t\to\infty$, where $(\bar\xi,\bar\xi',Z)$ are mutually independent and $\bar\xi, \bar\xi'$ both have law \eqref{eq:law-xi}.
\end{rem}

We are now ready to prove Proposition \ref{P:split2}. However, let us first give a sketch proof, which (roughly) describes the main steps of the argument (although some details look slightly different in the final version, in order to deal with technicalities that arise).

 \begin{proof}[Sketch proof of Proposition \ref{P:split2}.] 
	For $v,w\in \Omega$, let $\tau_{v,w}$ be the split time of $v$ and $w$. That is, the death time of the most recent common ancestor of $v$ and $w$.
	\begin{itemize}
		\item We first show that the asymptotic probability of selecting $w\in \cN_{at}$ and $v\in \cN_t$ with $w\preceq v$ is zero. Roughly speaking, this is because the probability that any specific individual at time $at$ has \emph{any} descendants at time $t$, tends to $0$ as $t\to \infty$.  This means that to obtain the asymptotic law of the split time, it is enough to provide an asymptotic for $$\mathbb{P}_{\delta_x}\bigg[  \sum_{w\in \cN_{at}, v\in \cN_{t}, w\npreceq v} \frac{1}{N_{at} \hat{N}^w_t} F\big(\frac{\tau_{v,w}}{t}\big)\, \bigg| \, N_t>0\bigg]$$ when $F$ is an arbitrary bounded continuous function, and where for $w\in \cN_{at}$, $\hat{N}^w_t$ is the size of the population at time $t$ \emph{without} counting the descendants of $w$. 
		\emph{(This roughly corresponds to Step 1 in the full proof below).}
		
		\item The many-to-two lemma at times $at,t$ allows the above expectation to be written as 
		$$	\frac{\varphi(x)}{\mathbb{P}_{\delta_x}(N_t>0)}\mathbb{Q}_{\delta_x}^2\left[ F(\frac{\tau}{t})\mathbf{1}_{\{\tau\le at\}} \frac{\varphi(\xi^1_{\tau-})}{\varphi(\xi^1_t)\varphi(\xi^2_{at})} \frac{1
		}{N_{at}\hat{N}_{t}} \e^{\int_0^\tau \beta(\xi^1_s)(\sm_2(\xi^1_s)-\sm_1(\xi^1_s)) \, \d s}
		\right]$$
		where $\tau$ is the split time of the two spines under $\mathbb{Q}_{\delta_x}^{2}$ and $\hat{N}_t$ is the population size at time $t$, not counting descendants of the second spine at time $at$. \emph{(This roughly corresponds to Step 3 in the full proof below).}
		
		\begin{figure}[h]
			\centering
			\includegraphics[width=.5\textwidth]{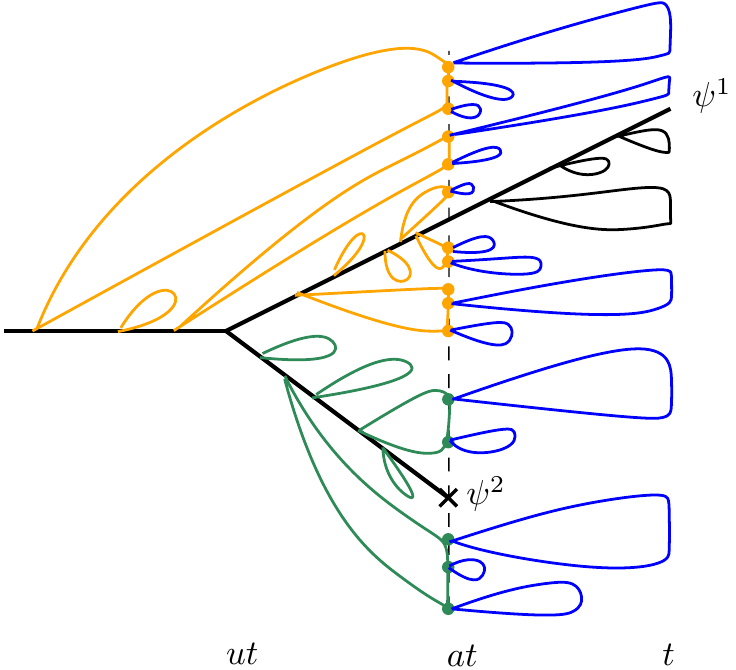}
			\caption{Suppose the two spines split from each other at time $\tau=ut$. The population at time $at$ can be broken up into those individuals that have branched off the first spine before time $at$ (depicted in orange) and those individuals that have branched off the second spine between times $ut$ and $at$ (depicted in green). %Under $\hat{\mathbb{Q}}^2_{\delta_x}(\cdot|\tau=ut)$ the size of the orange population is asymptotically $t$ times $aZ$ where $Z\sim \mathtt{Gamma}(2,2/\Sigma\langle 1, \tilde\varphi \rangle )$, and the size of the green population is asymptotically independent, and asymptotically $t$ times $(a-u)Z'$ where $Z\sim \mathtt{Gamma}(2,2/\Sigma\langle 1, \tilde\varphi \rangle)$. 
				Given the population at time $at$, the size of the population at time $t$ (without the descendants of the second spine, that is, $\hat{N}_t$) can again be broken up into two subpopulations: those that branch off the first spine between times $at$ and $t$ (depicted in black) and those that are descendants of non-spine particles at time $at$ (depicted in blue). 
				%The number of particles at time $at$ that have descendants at time $t$ is asympotically Poisson distributed (conditional on the size of the population at time $at$), and for each particle that does, the number of descendants at time $t$ is asympotically like $t$ times a $ \mathtt{Gamma}(1,2/\Sigma\langle 1, \tilde\varphi \rangle)$ random variable.
			}
		\label{fig:sketch}
		\end{figure}
		
		\item Next we consider $\hat{\mathbb{Q}}_{\delta_x}^2$ obtained by reweighting $\mathbb{Q}_{\delta_x}^2$ by $$(\beta(\xi_{\tau^- }^1)(\sm_2(\xi_{\tau^- }^1)-\sm_1(\xi_{\tau^- }^1)))^{-1} \e^{\int_0^\tau \beta(\xi^1_s)(\sm_2(\xi^1_s)-\sm_1(\xi^1_s)) \d s -\tau}.$$ 
		This change of measure alters the rate at which the spine particles split into two distinct spines (from rate $\beta(\sm_2 - \sm_1)$ to rate $1$) but doesn't affect the rate at which branching events occur that don't result in the spines splitting.
		Combining this with a change of variables and conditioning on $\tau$ (which has an exponential $1$ distribution under $\hat{\mathbb{Q}}^2_{\delta_x}$), we rewrite our expectation again, as 
		$$	\frac{\varphi(x)}{t\mathbb{P}_{\delta_x}(N_t>0)}\int_0^a \d u F(u) \hat{\mathbb{Q}}_{\delta_x}^2\left[ \frac{
			\beta(\xi^1_{ut}) \varphi(\sm_2(\xi^1_{ut})-\sm_1(\xi^1_{ut}))
		}{\varphi(\xi^1_t)\varphi(\xi^2_{at})} \frac{t^2}{N_{at}\hat{N}_{t}} \, \bigg| \, \tau=ut
		\right].$$
		Here the law $\hat{\mathbb{Q}}_{\delta_x}^2\left[\cdot | \tau=ut\right]$ makes rigorous sense: the system has a single spine and in fact evolves as under $\mathbb{Q}_{\delta_x}^1$ until time $ut$, where some (biased) branching event occurs, two spines are selected, and each of these initiates an independent $\mathbb{Q}^1$ process. \emph{(This roughly corresponds to Step 4 in the full proof below).}
		\item Now, we know by Lemma \ref{L:Yaglom} that $\varphi(x)(t\mathbb{P}_{\delta_x}(N_t>0))^{-1} \to \Sigma/2$ as $t\to \infty$. Moreover, under $\hat{\mathbb{Q}}^2_{\delta_x}$, similar arguments to those in the proof of Lemma \ref{L:Yaglom} (in particular, due to ergodicity of the spine motion) imply that the positions $\xi_{ut}^1, \xi_t^1, \xi_{at}^2$ of the spines are asymptotically independent of each other and of $N_{at}, \hat{N}_t$ as $t\to \infty$, with limiting laws described by $P(\xi \in A) = \langle \mathbf{1}_A\varphi, \tilde\varphi \rangle$ for $A\subset E$. 
		
		Furthermore, the limiting law of $N_{at}/t$ as $t\to \infty$ is described by $aZ+(a-{u})Z'$, where $(Z,Z')$ are a pair of independent $\mathtt{Gamma}(2,\, {2}/{\Sigma\langle 1,\tilde{\varphi}\rangle })$ random variables; this is because of the explicit description of the process under $\hat{\mathbb{Q}}_{\delta_x}^2(\cdot| \tau=ut)$ and item (c) of Lemma \ref{L:Yaglom}. In Figure \ref{fig:sketch}, $aZ$ and $(a-u)Z'$ correspond to the sizes of the orange and green populations respectively (after rescaling by $t$).

		Finally, the conditional limiting law of $\hat{N}_t/t$ given $N_{at}/t$ is that of a Gamma random variable with parameter $(2+K,\, {2}/{ \Sigma(1-a)\langle 1,\tilde{\varphi}\rangle })$, where $K\sim \mathtt{Poisson}({2N}/{(1-a)\Sigma \langle 1,\tilde{\varphi}\rangle } N)$ is itself random. This is because, given the collection of particles alive at time $at$, the first spine particle will have a number of offspring at time $t$ which is asymptotically like $t$ times a $\mathtt{Gamma}(2,{2}/{\Sigma(1-a)\langle 1,\tilde{\varphi}\rangle })$ random variable (Lemma \ref{L:Yaglom}(c) again; this corresponds to the population depicted in black in Figure \ref{fig:sketch}). Then, independently, each of the non-spine particles alive will have some descendant alive at time $t$ with probability asymptotically proportional to $t^{-1}$ times $\varphi$ of their positions. Using (essentially) the Poisson approximation of the binomial distribution, this results in a total number of non-spine particles with some descendant alive at time $t$ having asymptotic conditional distribution given by a $\mathtt{Poisson}({2N}/{(1-a)\Sigma \langle 1,\tilde{\varphi}\rangle } N)$ random variable. By Lemma \ref{L:Yaglom} (b), the number of offspring of each of these will approximately $t$ times an independent $\mathtt{Exponential}({2}/{\Sigma(1-a)\langle 1,\tilde{\varphi}\rangle })$, that is, a $\mathtt{Gamma}(1, {2}/{ \Sigma(1-a)\langle 1,\tilde{\varphi}\rangle})$ random variable. This corresponds to the population depicted in blue in Figure \ref{fig:sketch}. The additivity property of independent Gamma distributions completes the argument. \emph{(This roughly corresponds to Step 5 in the full proof below).}
		\item Plugging these asymptotics into the $\hat{\mathbb{Q}}_{\delta_x}^2$ expectation, and performing some simple explicit computations, we obtain the desired formula. \emph{(This roughly corresponds to Steps 6 and 7 in the full proof below).}
\end{itemize}
\end{proof}

\begin{proof}[Proof of Proposition \ref{P:split2}]
	Fix $0<a<1$ and for $w\in \cN_{at}$ and $v\in \cN_t$, write $\tau_{v,w}$ for the % {\color{red} death time of the most recent common ancestor of $v$ and $w$ (equiv. the last time that the genealogies of  $v$ and $w$ overlap).} 
	split time of $v$ and $w$ (as in the sketch proof).	It suffices to show that, for each continuous $F:[0,\infty)\to [0,1]$,
	\begin{equation}\label{E:m22goal1} \mathbb{P}_{\delta_x}\left[ \frac{1}{N_{at}N_t} \sum_{v\in \mathcal{N}_t, w\in \mathcal{N}_{at}} F({\tau_{v,w}}/{t})\, \bigg| \, N_t>0 \right] \to 
		\int_0^a F(u) f_a(u) \d u \end{equation}
	as $t\to \infty$.

	\medskip
	{\bf Step 1} In order to apply the many-to-two lemma, we first need to write the left-hand side of \eqref{E:m22goal1} in a slightly different form (that is asymptotically equivalent). We claim that
	\begin{equation}\label{E:m22split} 
		\mathbb{P}_{\delta_x}\bigg[ \frac{1}{N_{at}N_t} \sum_{v\in \mathcal{N}_t, w\in \mathcal{N}_{at}} F({\tau_{v,w}}/{t})% \mathbf{1}_{A^{\delta}_{v,w}}
		\, \bigg| \, N_t>0 \bigg] \sim \mathbb{P}_{\delta_x}\bigg[  \sum_{\substack{v\in \mathcal{N}_t, w\in \mathcal{N}_{at} \\ w\npreceq v}} \frac{%\mathbf{1}_{A^{\delta}_{v,w}}
		1}{N_{at}\hat{N}_t^w}F({\tau_{v,w}}/{t}) \, \bigg| \, N_t>0 \bigg]\end{equation}
	as $t\to \infty$ where $\hat{N}_t^w=N_t-N_t^w$ and $N_t^w=|\{v\in \mathcal{N}_t: w\preceq v\}|$.

\medskip

To prove this claim, let us consider the difference of the two quantities in the above asymptotic. Recalling that $F\in[0,1]$, we have
\begin{align*}
&\bigg|\mathbb{P}_{\delta_x}\bigg[ \frac{1}{N_{at}N_t} \sum_{v\in \mathcal{N}_t, w\in \mathcal{N}_{at}} F({\tau_{v,w}}/{t})
		\, \big| \, N_t>0 \bigg] 
		- \mathbb{P}_{\delta_x}\bigg[  \sum_{\substack{v\in \mathcal{N}_t, w\in \mathcal{N}_{at} \\ w\npreceq v}} \frac{%\mathbf{1}_{A^{\delta}_{v,w}}
		1}{N_{at}\hat{N}_t^w}F({\tau_{v,w}}/{t}) \, \big| \, N_t>0 \bigg] \bigg| \\
		&\le \bigg| \mathbb{P}_{\delta_x}\bigg[\sum_{\substack{v\in \mathcal{N}_t, w\in \mathcal{N}_{at} \\ w\preceq v}} \frac{F(\tau_{v, w}/t)}{N_t N_{at}}\big| N_t > 0 \bigg]\bigg| 
		+ \bigg| \mathbb{P}_{\delta_x}\bigg[\sum_{{\substack{v\in \mathcal{N}_t, w\in \mathcal{N}_{at} \\ w\npreceq v}}} \frac{F(\tau_{v, w}/t)}{N_{at}}(\frac{1}{\hat{N}_t^w} - \frac{1}{N_t})\big| N_t > 0\bigg]\big| \\
		& \le \mathbb{P}_{\delta_x}\bigg[\sum_{v \in \cN_t}\frac{1}{N_tN_{at}}\big| N_t > 0\bigg] + \mathbb{P}_{\delta_x}\bigg[\frac{1}{N_{at}}\sum_{w \in \cN_{at}}\frac{N_t^w}{N_t}\big| N_t > 0 \bigg]\\
		&= 2 \mathbb{P}_{\delta_x}\bigg[\frac{1}{N_{at}}\bigg| N_t > 0\bigg],
\end{align*}
where we have used the simple decomposition $\textstyle \sum_{w \in \cN_{at}}{N_t^w} ={N_t}$ in the final equality.
Note that, for arbitrarily large $M>0$,
\begin{align}
\mathbb{P}_{\delta_x}\bigg[\frac{1}{N_{at}}\bigg| N_t > 0\bigg]&\leq 
\frac{1}{M}\mathbb{P}_{\delta_x}\bigg[N_{at}\geq M \bigg| N_t > 0\bigg]+ 
\mathbb{P}_{\delta_x}\bigg[N_{at}\leq M \bigg| N_t > 0\bigg].
\label{2terms}
\end{align}
The first term on the right-hand side of \eqref{2terms} can be made 
arbitrarily small by taking $M$ sufficiently large. Moreover, the second term on the right-hand side of \eqref{2terms}
 converges to zero by Lemma \ref{L:Yaglom}.

\medskip

 It therefore suffices to prove that
 \begin{equation}\label{E:m22goal2}
\mathbb{P}_{\delta_x}\bigg[  \sum_{\substack{v\in \mathcal{N}_t, w\in \mathcal{N}_{at} \\ w\npreceq v}} \frac{%\mathbf{1}_{A^{\delta}_{v,w}}
1}{N_{at}\hat{N}_t^w}F({\tau_{v,w}}/{t}) \, \bigg| \, N_t>0 \bigg] \to \int_0^a F(u) f_a(u) \d u, \quad \text{as } t \to \infty
 \end{equation}
 This will be the new goal for the remainder of the proof.
	
	\medskip
	
	{\bf Step 2} 
In order to apply some bounded convergence results, it is convenient to define the following event for $\delta>0$. Namely, we write $A^\delta_{v,w}$ for the event that \[ \varphi(X_v(t))\ge \delta \text{ , } \varphi(X_w(at))\ge \delta \text{ , } \varphi(X_{v}(\tau_{v,w}^-))=\varphi(X_w(\tau_{v,w}^-))\ge\delta \text{ , }  \frac{N_{at}}t\ge \delta \text{ and } \frac{\hat{N}_t^w}t\ge \delta.\]
	We claim that it suffices to show that for each $\delta>0$ and continuous $F:[0,\infty)\to [0,1]$,
		\begin{equation}\label{E:m22goal} \mathbb{P}_{\delta_x}\bigg[ \sum_{\substack{v\in \mathcal{N}_t, w\in \mathcal{N}_{at} \\ w\npreceq v}} \frac{\mathbf{1}_{A^\delta_{v,w}}}{N_{at}\hat{N}^w_t}  F({\tau_{v,w}}/{t})\, \bigg| \, N_t>0 \bigg] \to \frac{c_\delta}{\langle 1, \tilde{\varphi}\rangle^2\Sigma} 
			\int_0^a F(u) f_a^\delta(u) \d u \end{equation}
	as $t\to \infty$, where 
		\begin{equation}\label{c-delta} 
		c_\delta:=\langle {\mathbf{1}_{\{\varphi\ge \delta\}}}, \tilde{\varphi}\rangle^2
		\langle \beta\varphi^2 (\sm_2-\sm_1)\mathbf{1}_{\{\varphi\ge \delta\}}, \tilde{\varphi}\rangle
		\end{equation}
		and for some $f_a^\delta(u)\nearrow f_a(u)$ as $\delta\searrow 0$, pointwise on $[0,a]$. 
		To see why \eqref{E:m22goal} suffices, note that using the definitions of $\mathbb{V}$ and $\sm_k$, 
		\[
		  \langle \beta \varphi^2(\sm_2-\sm_1)\mathbf{1}_{\{\varphi\ge \delta\}}, \tilde{\varphi}\rangle
		  = \langle \beta \mathbb{V}[\varphi]\mathbf{1}_{\{\varphi\ge \delta\}},  \tilde{\varphi}\rangle.
		\]
		Thus, using the boundedness of $\tilde\varphi$, $\varphi$, $\beta$ and Assumption \ref{A:Yaglom}.2, it follows that $c_\delta \uparrow \langle 1, \tilde{\varphi}\rangle^2\Sigma$ as $\delta\downarrow 0$.
		Moreover, since we know by Remark \ref{R:density1} that $f_a(u)$ integrates to $1$ over $[0,a]$, we can take $F\equiv 1$ in \eqref{E:m22goal} to see that  
	\[ \lim_{\delta\to 0} \lim_{t\to \infty }\mathbb{P}_{\delta_x}\bigg[ \sum_{\substack{v\in \mathcal{N}_t, w\in \mathcal{N}_{at} \\ w\npreceq v}} \frac{\mathbf{1}_{(A^\delta_{v,w})^c}}{N_{at}\hat{N}^w_t}  \, \bigg| \, N_t>0 \bigg] = \lim_{\delta\to 0} \lim_{t\to \infty } \bigg(1- \mathbb{P}_{\delta_x}\bigg[ \sum_{\substack{v\in \mathcal{N}_t, w\in \mathcal{N}_{at} \\ w\npreceq v}} \frac{\mathbf{1}_{A^\delta_{v,w}}}{N_{at}\hat{N}^w_t}  \, \bigg| \, N_t>0 \bigg]\bigg) = 0. \]
	Thus, given \eqref{E:m22goal}, we can take $t\to \infty$ and then $\delta\downarrow 0$ to deduce that the right-hand side of \eqref{E:m22split} converges to the right-hand side of \eqref{E:m22goal1} as $t\to \infty$.

	\medskip

	The remaining steps will focus on the proof of \eqref{E:m22goal} for an appropriate choice of $f_a^\delta$ and with $c_\delta$ defined in \eqref{c-delta}.
	
		\medskip
	{\bf Step 3} %Let us now move on to the proof of \eqref{E:m22goal}. %Because $N_{at}, \hat{N}_t^w, \tau_{v,w}$ are \ellen{XXX}-measurable, w
	We may apply the many-to-two formula, Lemma \ref{lem-many2few}, with $s_1=t, s_2=at$, to write
	\begin{align}\label{E:m22split_2}
		\mathbb{P}_{\delta_x}&\bigg[  \sum_{v\in \mathcal{N}_t, w\in \mathcal{N}_{at}, w\npreceq v} \frac{\mathbf{1}_{A^{\delta}_{v,w}}
		}{N_{at}\hat{N}_t^w}F({\tau_{v,w}}/{t}) \, \bigg| \, N_t>0 \bigg] \nonumber \\ = &	\frac{\varphi(x)}{\mathbb{P}_{\delta_x}(N_t>0)}\mathbb{Q}_{\delta_x}^2\left[ F({\tau}/{t})\mathbf{1}_{\{\tau\le at\}} \frac{\varphi(\xi^1_{\tau-})}{\varphi(\xi^1_t)\varphi(\xi^2_{at})} \frac{\mathbf{1}_{A^\delta_t}
	}{N_{at}\hat{N}_{t}} \e^{\int_0^\tau \beta(\xi^1_s)(\sm_2(\xi^1_s)-\sm_1(\xi^1_s)) \, \d s}
		\right]
	\end{align}
	where $\tau=\tau_{\psi_t^1,\psi_{at}^2}$, $\hat{N}_t=\hat{N}_t^{\psi_{at}^2}$, and $A^\delta_t=A^\delta_{\psi^1_t,\psi^2_{at}}$. The application of Lemma \ref{m2k-ktimes} is justified since $$Y = F\big({\tau}/{t}\big)\mathbf{1}_{\{\tau \le at\}}\frac{\mathbf{1}_{A^\delta_t}}{N_{at}\hat{N}_t}$$ is $\cF_{ (at, t)}^2$- measurable.
	
	\medskip

	{\bf Step 4} Next we consider $\hat{\mathbb{Q}}_{\delta_x}^2$ obtained by reweighting $\mathbb{Q}_{\delta_x}^2$ by $$\frac{1}{\beta(\xi_{\tau^- }^1)(\sm_2(\xi_{\tau^- }^1)-\sm_1(\xi_{\tau^- }^1))} \e^{\int_0^\tau \beta(\xi^1_s)(\sm_2(\xi^1_s)-\sm_1(\xi^1_s)) \d s -\tau}.$$ 
	This change of measure alters the rate at which the spine particles split into two distinct spines (from rate $\beta(\sm_2 - \sm_1)$ to rate $1$). Note, however, that it doesn't affect the rate at which branching events occur that don't result in the spines splitting.
	 Combining this with a change of variables and conditioning on $\tau$, it follows that the right-hand side of \eqref{E:m22split_2} is equal to 
	\begin{equation}\label{E:m22split_3}
		\frac{\varphi(x)}{t\mathbb{P}_{\delta_x}(N_t>0)}\int_0^a \d u F(u) \hat{\mathbb{Q}}_{\delta_x}^2\left[ \frac{
			\beta(\xi^1_{ut}) \varphi(\xi^1_{ut})(\sm_2(\xi^1_{ut})-\sm_1(\xi^1_{ut}))\mathbf{1}_{A^\delta_t}
		}{\varphi(\xi^1_t)\varphi(\xi^2_{at})} \frac{t^2}{N_{at}\hat{N}_{t}} \, \bigg| \, \tau=ut
		\right].
	\end{equation}
	Note that under $\hat{\mathbb{Q}}_{\delta_x}^2(\cdot | \tau = ut)$, the process behaves as follows.
	\begin{itemize} 
	\item Until time $ut$, the process moves with biased motion as in \eqref{k-motion-bias} with $\zeta$ given by \eqref{specific-zeta}. At rate $\beta \sm_1$ branching events occur, at which point, the offspring distribution is given by $\mathcal{P}^{2, 1}$ (as in \eqref{bias-offspring}) and the $i$-th particle is chosen to be the spine with probability proportional to $\varphi(x_i)$.
	\item At time $ut$ a branching event occurs, where the law of the offspring is given by $\mathcal{P}^{2, 2}$ and particles
	$i, j$ with $i\neq j$ are chosen as the two (distinct) spines with probability proportional to $\varphi(x_i)\varphi(x_j)$.	
	\item After time $ut$, the processes issued from the two spine particles evolve under $\mathbb{Q}^1$ and those issued from the non-spine particles evolve under $\mathbb{P}$.
	\end{itemize}
	Note also that by Lemma \ref{L:Yaglom},  $\varphi(x)/(t\mathbb{P}_{\delta_x}(N_t>0))\to \Sigma/2$ as $t\to \infty$. 
	Thus, if we write $\hat{\mathbb{Q}}^2_{\delta_x,ut}(\cdot):=\hat{\mathbb{Q}}_{\delta_x}^2(\cdot \, | \, \tau=ut)$, we have that 
	
	\begin{align}
		\mathbb{P}_{\delta_x}&\Bigg[  \sum_{\substack{v\in \mathcal{N}_t, \, w\in \mathcal{N}_{at} \\ w\npreceq v}} \frac{\mathbf{1}_{A^{\delta}_{v,w}}
		}{N_{at}\hat{N}_t^w}F({\tau_{v,w}}/{t}) \, \bigg| \, N_t>0 \Bigg]\notag\\
		&\sim \frac{\Sigma}{2} \int_0^a \d u F(u) \hat{\mathbb{Q}}_{\delta_x,ut}^2\bigg[ \frac{\beta(\xi^1_{ut}) \varphi(\xi^1_{ut})(\sm_2(\xi^1_{ut})-\sm_1(\xi^1_{ut}))}{\varphi(\xi^1_t)\varphi(\xi^2_{at})} \frac{t^2
			\mathbf{1}_{A_t^\delta}
		}{N_{at}\hat{N}_{t}}
		\bigg],
\label{integral_asymptotic}
	\end{align}
	as $t\to \infty$. 
	
	\medskip
	{\bf Step 5}
Our next goal is to describe the limit in law of $(\xi_{ut}^1, \xi_{t}^1,\xi_{at}^2, N_{at}/t, \hat{N}_t/t)$ under $\hat{\mathbb{Q}}_{x,ut}^2$ as $t\to \infty$. More precisely, we claim that it converges to $(\bar\xi, \bar\xi', \bar\xi'', N, \hat{N})$ where 
\begin{itemize} \item $\bar\xi, \bar\xi', \bar\xi''$ are independent of eachother and of $(N,\hat{N})$, each with law given by \eqref{eq:law-xi};
	\item the law of $N$ is that of  $a Z+(a-u)Z'$, where $(Z,Z')$ are a pair of independent $\mathtt{Gamma}(2, \, 2/ \Sigma\langle 1,\tilde{\varphi}\rangle  )$ random variables;
	\item  conditionally on $N$, the law of $\hat{N}$ is that of a $\mathtt{Gamma}(2+K,\,2/ \Sigma(1-a)\langle 1,\tilde{\varphi}\rangle )$ random variable with random $K\sim \mathtt{Poisson}({2N}/{(1-a)\Sigma \langle 1,\tilde{\varphi}\rangle } N)$.
\end{itemize}
To justify this claim, we identify the limiting Laplace transform of $(\xi_{ut}^1, \xi_{t}^1,\xi_{at}^2, N_{at}/t, \hat{N}_t/t)$. To this end, for arbitrary $\theta, \mu, \eta, \rho,\chi\ge 0$, let us consider
\begin{multline}\label{E:Laplace5}
	\hat{\mathbb{Q}}_{x,ut}^2\big[{\rm e}^{-\theta \xi_{ut}^1}{\rm e}^{-\mu \xi_t^1}{\rm e}^{ - \eta \xi_{at}^2}{\rm e}^{- \rho N_{at}/t}{\rm e}^{-\chi {\hat N_t}/{t}} \big] \\= \hat{\mathbb{Q}}_{x,ut}^2\big[{\rm e}^{-\theta \xi_{ut}^1- \eta \xi_{at}^2 - \rho {N_{at}}/{t}} \, \hat{\mathbb{Q}}_{x,ut}^2 \big[{\rm e}^{-\mu\xi_t^1}{\rm e}^{-\chi {\hat N_t}/{t}}\, | \, \mathcal{F}_{at}^2 \big] \big],
	\end{multline}
where $\mathcal{F}_{at}^2$ is the $\sigma$-algebra containing all the information about the process, including the spines, up to time $at$.

\medskip

Recalling the description of the process under $\hat{\mathbb{Q}}_{x,ut}^2$, we see that 
\begin{align} & \hat{\mathbb{Q}}_{\delta_x,ut}^2 \big[{\rm e}^{-\mu\xi_t^1}{\rm e}^{-\chi \hat{N}_t/t}\, | \, \mathcal{F}_{at}^2 \big] = \mathbb{Q}^1_{\delta_{\xi_{at}^1}}\big[{\rm e}^{-\mu \xi_{(1-a)t}}{\rm e}^{-\chi {N_{(1-a)t}}/{t}} \big] \prod_{\substack{v\in \cN_{at} \\ v\ne \psi_{at}^1,\psi_{at}^2}} \mathbb{P}_{\delta_{X_v(at)}}\big[{\rm e}^{-\chi {N_{(1-a)t}}/{t}} \big] \nonumber \\
	& =  \mathbb{Q}^1_{\delta_{\xi_{at}^1}}\big[{\rm e}^{-\mu \xi_{(1-a)t}}{\rm e}^{-\chi{N_{(1-a)t}}/{t}} \big] \exp\bigg( \sum_{\substack{v\in \cN_{at} \\ v\ne \psi_{at}^1,\psi_{at}^2}} \log\big(1-(1-\mathbb{P}_{\delta_{X_v(at)}}[{\rm e}^{-\chi {N_{(1-a)t}}/{t}} ])\big) \bigg) \nonumber \\
	& =  \mathbb{Q}^1_{\delta_{\xi_{at}^1}}\big[{\rm e}^{-\mu \xi_{(1-a)t}}{\rm e}^{-\chi {N_{(1-a)t}}/{t}} \big] (1-E_t) \, \exp\bigg( \sum_{\substack{v\in \cN_{at} \\ v\ne \psi_{at}^1,\psi_{at}^2}} -(1-\mathbb{P}_{\delta_{X_v(at)}}[{\rm e}^{-\chi {N_{(1-a)t}}/{t}} ]) \bigg) 	\label{E:Qdecomp} \end{align}
where 
\[ E_t := 1-\exp\bigg(\sum_{\substack{v\in \cN_{at} \\ v\ne \psi_{at}^1,\psi_{at}^2}} \log\big(1-(1-\mathbb{P}_{\delta_{X_v(at)}}[{\rm e}^{-\chi {N_{(1-a)t}}/{t}} ])\big)+(1-\mathbb{P}_{\delta_{X_v(at)}}[{\rm e}^{-\chi {N_{(1-a)t}}/{t}} ])\bigg).
\]
We claim that $E_t$ belongs to $
 [0,({cN_{at}}/{t^2})\wedge 1] $
for some absolute deterministic constant $c$, by Lemma \ref{L:Yaglom}(a). 
The lower bound of zero and the upper bound of 1 are  obvious. To see where the upper bound of ${cN_{at}}/{t^2} $comes from, note that there are at most $N_{at}$ elements of the sum defining $E_{t}$,  $\log(1-x)+x\leq -x^2/2$ and 
$\textstyle 1-\mathbb{P}_{\delta_{X_v(at)}}[{\rm e}^{-\chi {N_{(1-a)t}}/{t}} ]\leq \chi\sup_{x\in E}\mathbb{P}_{\delta_{x}}[ 
N_{(1-a)t} ]/t\leq K/t$, for some appropriate $K>0$ (where this final inequality follows from criticality and Assumption \ref{A:Yaglom}).

\medskip

Note also that by Lemma \ref{L:Yaglom}(c), 
\[ \mathbb{Q}^1_{\delta_{\xi_{at}^1}}\big[{\rm e}^{-\mu \xi_{(1-a)t}}{\rm e}^{-\chi{N_{(1-a)t}}/{t}} \big]=s^2 \langle \varphi{\rm e}^{-\mu \cdot}, \tilde{\varphi}\rangle (1+e_t(\xi_{at}^1))\]
where $s:=(1+\chi {\Sigma(1-a) \langle 1,\tilde \varphi\rangle }/{2})^{-1}<1$ and  $e_t(x)$ is such that $e_t(x)\to0$ in the pointwise sense on $E$ and
\[\sup_{ t,x}|e_t(x)|<\infty.\] 
Let us further denote for $x\in E, t\ge 0$,
\[ g_t(x):=t\big(1-\mathbb{P}_{\delta_x}[{\rm e}^{-\chi {N_{(1-a)t}}/{t}}]\big)>0\]
so that by Lemma \ref{L:Yaglom}, 
$\sup\nolimits_{t,x} g_t(x) <\infty $ and 
\begin{align*}
 g_t(x) = t\mathbb{P}_{\delta_x}(N_{(1-a)t}>0) \big(1-\mathbb{P}_{\delta_x}({\rm e}^{-\chi {N_{(1-a)t}}/{t} 
 } \, | \, N_{(1-a)t}>0
 )\big) \to \frac{2(1-s)}{\Sigma(1-a)} \varphi(x)  
 \end{align*} 
 pointwise on $E$ as $t \to\infty$.
Using this notation in the right hand side of \eqref{E:Qdecomp}, we see that 
\begin{equation*}
	\hat{\mathbb{Q}}_{\delta_x,ut}^2 \big[{\rm e}^{-\mu\xi_t^1}{\rm e}^{-\chi \hat{N}_t/t}\, | \, \mathcal{F}_{at}^2 \big] = s^2 \langle \varphi{\rm e}^{-\mu \cdot}, \tilde{\varphi}\rangle{\rm e}^{-\frac{X_{at}[g_t]}{t}} (1+e_t(\xi_{at}^1))(1-E_t)
\end{equation*}
so that 
\begin{multline}\label{E:bounded}
		\hat{\mathbb{Q}}_{x,ut}^2\big[{\rm e}^{-\theta \xi_{ut-}^1}{\rm e}^{-\mu \xi_t^1}{\rm e}^{ - \eta \xi_{at}^2}{\rm e}^{- \rho {N_{at}}/{t}}{\rm e}^{-\chi {\hat N_t}/{t}} \big] \\ = s^2 \langle \varphi{\rm e}^{-\mu \cdot}, \tilde{\varphi}\rangle \hat{\mathbb{Q}}_{x,ut}^2\big[{\rm e}^{-\theta \xi_{ut-}^1- \eta \xi_{at}^2 - \rho {N_{at}}/{t}}{\rm e}^{-{X_{at}[g_t]}/{t}} (1+e_t(\xi_{at}^1))(1-E_t) \big].
\end{multline}
Now, we claim that under $\hat{\mathbb{Q}}_{x,ut}^2$, \begin{equation}\label{E:6bitsconverge}(\xi_{ut}^1, \xi_{at}^2, e_t(\xi_{at}^1), N_{at}/t, X_t[g_t]/t, E_t)\Rightarrow (\bar \xi, \bar \xi'', 0, N, \tfrac{2(1-s)}{\Sigma(1-a)\langle 1, \tilde\varphi \rangle } N, 0)\end{equation} as $t\to \infty$, where $(\bar\xi,\bar\xi'',N)$ have joint law as described in the bullet points at the start of Step 5. Since everything inside the expectation on right-hand side of \eqref{E:bounded} is deterministically bounded, \eqref{E:6bitsconverge} implies that 
\begin{multline*}
	\hat{\mathbb{Q}}_{x,ut}^2\big[{\rm e}^{-\theta \xi_{ut-}^1}{\rm e}^{-\mu \xi_t^1}{\rm e}^{ - \eta \xi_{at}^2}{\rm e}^{- \rho {N_{at}}/{t}}{\rm e}^{-\chi {\hat N_t}/{t}} \big] \to \\
	\langle \varphi{\rm e}^{-\theta \cdot}, \tilde{\varphi} \rangle\langle \varphi{\rm e}^{-\mu \cdot}, \tilde{\varphi} \rangle\langle \varphi{\rm e}^{-\lambda \cdot}, \tilde{\varphi} \rangle s^2 \int_0^\infty p_{a,u}(x){\rm e}^{-\rho x}{\rm e}^{\tfrac{2(s-1)}{(1-a)\Sigma (1,\tilde{\varphi})}x} \d x, 
\end{multline*}
where $p_{a,u}$ is the density of $N$ (with law as described in the second bullet point). It is easy to check using the explicit expressions for the Laplace transforms of Poisson and Gamma random variables, that the right-hand side above is exactly the joint Laplace transform of our desired limit $(\xi, \xi',\xi'', N, \hat{N})$. Thus, it only remains to justify \eqref{E:6bitsconverge}.

\medskip

To this end, let us write $\cN_{at}^1$ for the collection of particles alive at time $at$ that have branched off the first spine between times $0$ and $at$ (depicted in orange in Figure \ref{fig:sketch}). We also write $\cN_{at}^2$ for those particles alive at time $at$ that have branched off the second spine between times $ut$ and $at$ (depicted in green in Figure \ref{fig:sketch}). Then, using Lemma \ref{L:Yaglom} (and its extension Remark \ref{R:Yaglom}), it follows that 
\begin{equation}
(\xi_{ut}^1, e_t(\xi_{at}^1),\frac1t{|\cN_{at}^1|}, \frac{1}{t}{\sum_{v\in \cN_{at}^1} g_t(X_{v}(at))})\Rightarrow (\xi, 0, N', \tfrac{2(1-s)}{\Sigma(1-a)\langle 1, \tilde\varphi \rangle } N') 
\label{LHSabove}
\end{equation}
say, where $(\xi, N')$ are independent, $\xi$ has law given in \eqref{eq:law-xi} and $N'$ has the law of $a$ times a $\mathtt{Gamma}(2, 2/\Sigma \langle 1, \tilde\varphi\rangle)$ random variable (recall that $e_t(x)$ is deterministically uniformly bounded over $t,x$ and converges pointwise to $0$ on $E$). Now, given all the information in the quadruple displayed on the left-hand side of \eqref{LHSabove},  the (conditional) joint law of \[(\xi_{at}^2,\frac1t{|\cN_{at}^2|}, \frac{1}{t}{\sum_{v\in \cN_{at}^2} g_t(X_{v}(at))})\] is given by the $\mathbb{Q}_{\delta_X}^1$ law of $(\xi_{(a-u)t}^2, N_{(a-u)t}/t, X_{(a-u)t}[g_t]/t)$, where the (conditional) law of $X$ is explicit but not required here. Again, by Lemma \ref{L:Yaglom}, in particular part (c), 
\[
(\xi_{(a-u)t}^2, N_{(a-u)t}/t, X_{(a-u)t}[g_t]/t) \Rightarrow (\hat{\xi}, \hat{N},\tfrac{2(1-s)}{\Sigma(1-a)\langle 1, \tilde\varphi \rangle } \hat{N} )
\]
as $t \to \infty$ under $\mathbb{Q}_{\delta_X}^1$, where $\hat{\xi},\hat{N}$ are independent,  $\hat{\xi}$ has law given by \eqref{eq:law-xi}, and $\hat{N}$ has the law of $(a-u)$ times a $\mathtt{Gamma}(2,2/\Sigma\langle 1 ,\tilde{\varphi} \rangle))$ random variable, independently of the value of $X$. Putting these observations together, plus the fact that $E_t\in [0,cN_{at}/t^2]$ for some deterministic $c$, gives us \eqref{E:6bitsconverge}.
 
\medskip

{\bf Step 6}
Using the convergence in law {(and associated notation for limiting variables)} from Step 5, plus boundedness of the functionals in question, we deduce that for each $0\le u \le a$,
	\begin{align}
	 &\frac{\Sigma}{2} \int_0^a \d u F(u) \hat{\mathbb{Q}}_{\delta_x,ut}^2\bigg[ \frac{\beta(\xi^1_{ut}) \varphi(\xi^1_{ut})(\sm_2(\xi^1_{ut})-\sm_1(\xi^1_{ut}))}{\varphi(\xi^1_t)\varphi(\xi^2_{at})} \frac{t^2
			\mathbf{1}_{A_t^\delta}
		}{N_{at}\hat{N}_{t}}
		\bigg]\notag\\
		&{\to}\frac{\Sigma}{2} \int_0^a \d u F(u)\mathbb{E}\left[ \frac{\beta(\bar\xi) \varphi(\bar\xi)(\sm_2(\bar\xi)-\sm_1(\bar\xi))}{\varphi(\bar\xi')\varphi(\bar\xi'')} \mathbf{1}_{\{\varphi(\bar\xi)\ge\delta, \varphi(\bar\xi')\ge\delta, \varphi(\bar \xi'')\ge\delta\}} \frac{\mathbf{1}_{\{N\ge\delta, \hat{N}\ge \delta\}} }{N\hat{N}}\right], \label{eq:step6}
		\end{align}
		as ${t \to \infty}$.  The expectation in the integrand on the right-hand side of \eqref{eq:step6} is equal to
		\begin{align}\label{Qxu2}
%		&  \mathbb{E}\left[ \frac{(\beta \varphi(\sm_2-\sm_1))(\xi)}{\varphi(\xi')\varphi(\xi^2_{at})} \mathbf{1}_{\varphi(\xi)\ge\delta, \varphi(\xi')\ge\delta, \varphi(\xi'')\ge\delta} \frac{\mathbf{1}_{N\ge\delta, \hat{N}\ge \delta} }{N\hat{N}}\right] \nonumber \\
		&c_\delta \mathbb{E}\left[\frac{\mathbf{1}_{\{N\ge\delta\}}}{N} \mathbb{E}\left[\frac{\mathbf{1}_{\{\hat{N}\ge\delta\}}}{\hat{N}} \, \bigg| \, N\right] \right]\nonumber \\
		& = c_\delta \mathbb{E}\left[\frac{\mathbf{1}_{\{N\ge\delta\}}}{N} \left(\mathbb{E}\left[\frac{1}{\hat{N}} \, \bigg| \, N\right]- \mathbb{E}\left[\frac{\mathbf{1}_{\{\hat{N}<\delta\}}}{\hat{N}} \, \bigg| \, N\right]\right) \right]\nonumber \\
		& =		{c_\delta
		}\mathbb{E} \left[ \frac{\mathbf{1}_{\{N\ge\delta\}}}{N} \frac{1-\exp(-\tfrac{1}{1-a}(\tfrac{2}{\Sigma\langle 1,\tilde{\varphi}\rangle} N))}{N} \right]-c_\delta\mathbb{E} \left[ \frac{\mathbf{1}_{\{N\ge\delta\}}}{N}\mathbb{E}\left[\frac{\mathbf{1}_{\{\hat{N}<\delta\}}}{\hat{N}} \, \bigg| \, N\right] \right]\nonumber \\
	& =		{c_\delta
	}\mathbb{E} \left[  \frac{1-\exp(-\tfrac{1}{1-a}(\tfrac{2}{\Sigma\langle 1,\tilde{\varphi}\rangle} N))}{N^2} \right]-c_\delta\mathbb{E} \left[ \frac{\mathbf{1}_{\{N<\delta\}}}{N} \frac{1-\exp(-\tfrac{1}{1-a}(\tfrac{2}{\Sigma\langle 1,\tilde{\varphi}\rangle} N))}{N} \right]  \notag\\
		&\hspace{8cm}- c_\delta\mathbb{E} \left[ \frac{\mathbf{1}_{\{N\ge\delta\}}}{N}\mathbb{E}\left[\frac{\mathbf{1}_{\{\hat{N}<\delta\}}}{\hat{N}} \, \bigg| \, N\right] \right]\nonumber \\
		& =	:		\frac{4c_\delta
		}{a^2\Sigma^2\langle 1,\tilde{\varphi}\rangle^2
	} \mathbb{E} \left[ \frac{1-\exp(-\tfrac{a}{1-a}(\tfrac{2}{a\Sigma\langle 1,\tilde{\varphi}\rangle} N))}{(\tfrac{2}{a\Sigma\langle 1,\tilde{\varphi}\rangle }N)^2} \right] - h(\delta)
	\end{align}
	as $t\to \infty$, where  under $\mathbb{E}$, $(\bar\xi,\bar\xi',\bar\xi'',N,\hat{N})$ are described in Step 5 (recalling in particular the law of $\hat{N}$ given $N$) and $h(\delta)\ge 0$ {and $
		c_\delta:=\langle {\mathbf{1}_{\{\varphi\ge \delta\}}}, \tilde{\varphi}\rangle^2
		\langle \beta\varphi^2 (\sm_2-\sm_1)\mathbf{1}_{\{\varphi\ge \delta\}}, \tilde{\varphi}\rangle.
		$}

	\medskip
	
	{\bf Step 7} 
	Recall that by \eqref{E:m22goal} and \eqref{integral_asymptotic}, our aim is to prove that 
	\begin{equation*}
	\frac{\Sigma}{2} \int_0^a \d u F(u) \hat{\mathbb{Q}}_{\delta_x,ut}^2\bigg[ \frac{\beta(\xi^1_{ut}) \varphi(\xi^1_{ut})(\sm_2(\xi^1_{ut})-\sm_1(\xi^1_{ut}))}{\varphi(\xi^1_t)\varphi(\xi^2_{at})} \frac{t^2
			\mathbf{1}_{A_t^\delta}
		}{N_{at}\hat{N}_{t}}
		\bigg]
		\to
		\frac{c_\delta}{\langle 1, \tilde{\varphi}\rangle^2\Sigma} 
			\int_0^a F(u) f_a^\delta(u) \d u 
			\end{equation*}
	as $t\to \infty$, 
%	where 
%		$
%		c_\delta:=\langle {\mathbf{1}_{\{\varphi\ge \delta\}}}, \tilde{\varphi}\rangle^2
%		\langle \beta\varphi^2 (\sm_2-\sm_1)\mathbf{1}_{\{\varphi\ge \delta\}}, \tilde{\varphi}\rangle
%		$
%		and 
		for some $f_a^\delta(u)\nearrow f_a(u)$ as $\delta\searrow 0$, pointwise on $[0,a]$.
	
	\medskip

	First notice that $h(\delta)$ in \eqref{Qxu2}  converges to $0$ as $\delta\searrow 0$, since the law of $N\hat{N}$ has negative moments of all orders. Then
	\eqref{Qxu2}, 
\eqref{eq:step6} imply the result, since writing $Y={2N}/{a\Sigma(1,\tilde{\varphi})}$ (so that $Y\sim Y'+(1-\tfrac{u}{a})Y''$ for independent $Y', Y''\sim \mathtt{Gamma}(2,1)$) we have: 
	\begin{align*}  
	& \frac{2}{a^2} \mathbb{E} \left[ \frac{1-\exp(-\tfrac{a}{1-a}(\tfrac{2}{a\Sigma(1,\tilde{\varphi})} N))}{(\tfrac{2}{a\Sigma(1,\tilde{\varphi})}N)^2} \right] \\
	& = \frac{2}{a^2} \int_0^\infty \mathbb{E}\left[ \theta \exp(-\theta Y) - \theta \exp(-(\theta + \tfrac{a}{1-a})Y) \right] \, d\theta \\
		& = \frac{2}{a^2} \int_0^\infty \left(\frac{\theta}{(1+\theta)^2(1+(1-\tfrac{u}{a})\theta)^2} - \frac{\theta}{(1+\theta+\tfrac{a}{1-a})^2(1+(1-\tfrac{u}{a})(\theta+\tfrac{a}{1-a}))^2} \right)\d \theta \\
		& = f_a(u).
	\end{align*}
To calculate the integral in the penultimate line, we have used the change of variables $x = \theta + \frac{a}{1-a}$ for the second integrand, the fact that the anti-derivative of $\frac{y}{(1+y)^2(1+\gamma y)^2}$ is given by
\[
  \frac{1}{(\gamma - 1)^3}\left(\frac{(\gamma - 1)(\gamma y + y +2)}{(y+1)(\gamma y + 1)} - (\gamma + 1)\log\big(\frac{y + 1}{\gamma y + 1}\big) \right),
\]
and that the anti-derivative of $\frac{1}{(1+y)^2(1+\gamma y)^2}$ is given by 
\[
\frac{1}{(\gamma - 1)^3}\left(\frac{-(\gamma - 1)(2\gamma y + \gamma +1)}{(y + 1)(\gamma y+1)}  + 2\gamma\log\big(\frac{y + 1}{\gamma y + 1}\big) \right).
\]
The proof is now complete.
\end{proof}

\begin{proof}[Proof of Proposition \ref{prop:joint-conv}]
The proof of this proposition is contained in the proof of Step 5 above, ignoring the contribution from the second spine. 
\end{proof}

\paragraph{Acknowledgements}

{This work was partially supported by by EPSRC grant EP/W026899/1, UKRI Future Leader's Fellowship MR/W008513/1, and the “PHC Alliance” programme (project number: 47867UJ), funded by the French Ministry for Europe and Foreign Affairs, the French Ministry for Higher Education and Research and the UK Department for Business, Energy and Industrial Strategy.}

\bibliography{references}{}

\begin{thebibliography}{10}

\bibitem{AHbook}
S.~Asmussen and H.~Hering.
\newblock {\em Branching processes}, volume~3 of {\em Progress in Probability
  and Statistics}.
\newblock Birkh\"{a}user Boston, Inc., Boston, MA, 1983.

\bibitem{Bertoinbook}
J.~Bertoin.
\newblock {\em Random fragmentation and coagulation processes}, volume 102 of
  {\em Cambridge Studies in Advanced Mathematics}.
\newblock Cambridge University Press, Cambridge, 2006.

\bibitem{CR}
B.~Chauvin and A.~Rouault.
\newblock {KPP} equation and supercritical branching {B}rownian motion in the
  subcritical speed area. application to spatial trees.
\newblock {\em Probability theory and related fields}, 80(2):299--314, 1988.

\bibitem{CHO}
A.~Cortines, L.~Hartung, and O.~Louidor.
\newblock The structure of extreme level sets in branching {B}rownian motion.
\newblock {\em Ann. Probab.}, 47(4):2257--2302, 2019.

\bibitem{SNTE-III}
A.~M. Cox, E.~Horton, A.~E. Kyprianou, and D.~Villemonais.
\newblock Stochastic methods for neutron transport equation {III}:
  {G}enerational many-to-one and {$k_\mathtt{eff}$}.
\newblock {\em SIAM J. Appl. Math.}, 81(3):982--1001, 2021.

\bibitem{Janos}
J.~Engl\"{a}nder.
\newblock {\em Spatial branching in random environments and with interaction},
  volume~20 of {\em Advanced Series on Statistical Science \& Applied
  Probability}.
\newblock World Scientific Publishing Co. Pte. Ltd., Hackensack, NJ, 2015.

\bibitem{FRS}
F.~Foutel-Rodier and E.~Schertze.
\newblock Convergence of genealogies through spinal decomposition with an
  application to population genetics.

\bibitem{moments}
I.~Gonzalez, E.~Horton, and A.~E. Kyprianou.
\newblock Asymptotic moments of spatial branching processes.
\newblock {\em Probability Theory and Relatied Fields}, to appear.

\bibitem{HHK}
S.~C. Harris, M.~Hesse, and A.~E. Kyprianou.
\newblock Branching {B}rownian motion in a strip: survival near criticality.
\newblock {\em Ann. Probab.}, 44(1):235--275, 2016.

\bibitem{SNTE-II}
S.~C. Harris, E.~Horton, and A.~E. Kyprianou.
\newblock Stochastic methods for the neutron transport equation {II}: almost
  sure growth.
\newblock {\em Ann. Appl. Probab.}, 30(6):2815--2845, 2020.

\bibitem{YaglomNTE}
S.~C. Harris, E.~Horton, A.~E. Kyprianou, and M.~Wang.
\newblock Yaglom limit for critical non-local branching {M}arkov processes.
\newblock {\em Annals of Probability}, 50(6):2373--2408, 2022.

\bibitem{HJR}
S.~C. Harris, S.~G.~G. Johnston, and M.~I. Roberts.
\newblock The coalescent structure of continuous-time {G}alton-{W}atson trees.
\newblock {\em Ann. Appl. Probab.}, 30(3):1368--1414, 2020.

\bibitem{GWgenealogy}
S.~C. Harris, S.~G.~G. Johnston, and M.~I. Roberts.
\newblock The coalescent structure of continuous-time {G}alton-{W}atson trees.
\newblock {\em Ann. Appl. Probab.}, 30(3):1368--1414, 2020.

\bibitem{Many2few}
S.~C. Harris and M.~I. Roberts.
\newblock The many-to-few lemma and multiple spines.
\newblock {\em Ann. Inst. Henri Poincar\'{e} Probab. Stat.}, 53(1):226--242,
  2017.

\bibitem{HK}
E.~Horton and A.~E. Kyprianou.
\newblock {\em Stochastic neutron transport and non-local branching Markov
  processes}.
\newblock Probability and its Applications. Birkh\"auser, 2023.

\bibitem{SNTE-I}
E.~Horton, A.~E. Kyprianou, and D.~Villemonais.
\newblock Stochastic methods for the neutron transport equation {I}: linear
  semigroup asymptotics.
\newblock {\em Ann. Appl. Probab.}, 30(6):2573--2612, 2020.

\bibitem{LTZ}
E.~Lubetzky, C.~Thornett, and O.~Zeitouni.
\newblock Maximum of branching {B}rownian motion in a periodic environment.
\newblock {\em Ann. Inst. Henri Poincar\'{e} Probab. Stat.}, 58(4):2065--2093,
  2022.

\bibitem{ZS}
Z.~Shi.
\newblock {\em Branching random walks}, volume 2151 of {\em Lecture Notes in
  Mathematics}.
\newblock Springer, Cham, 2015.
\newblock Lecture notes from the 42nd Probability Summer School held in Saint
  Flour, 2012.

\end{thebibliography}
\bibliographystyle{plain}
\end{document}